\documentclass[12pt]{amsart}
\usepackage[mathcal]{eucal}
\usepackage{amssymb, graphicx, labelfig}
\usepackage[all]{xy}
\usepackage{hyperref}

\newtheorem{thm}{Theorem}
\newtheorem{prop}[thm]{Proposition}
\newtheorem{lem}[thm]{Lemma}

\newtheorem{cor}[thm]{Corollary}

\theoremstyle{remark}
\newtheorem{rem}[thm]{Remark}

\theoremstyle{definition}

\newcommand{\col}{\kern -3pt :}
\newcommand{\C}{\mathbb C}
\newcommand{\R}{\mathbb R}
\newcommand{\Z}{\mathbb Z}

\newcommand{\HH}{\mathbb H}

\newcommand{\Id}{\mathrm{Id}}
\newcommand{\End}{\mathrm{End}}

\newcommand{\SSS}{\mathcal S^A(S)}
\newcommand{\SSSS}{\mathcal S^{-1}(S)}

\newcommand{\TT}{\mathcal T^\omega(\lambda)}
\newcommand{\TTT}{\mathcal T^{\iota}(\lambda)}
\newcommand{\RR}{\mathcal R_{\SL(\C)}(S)}
\newcommand{\RS}{\mathcal R_{\PSL(\C)}^{\mathrm{Spin}}(S)}
\newcommand{\Spin}{\mathrm{Spin}(S)}
\newcommand{\ZZ}{\mathcal Z^\omega(\lambda)}
\newcommand{\ZZZ}{\mathcal Z^\iota(\lambda)}
\newcommand{\SL}{\mathrm{SL}_2}
\newcommand{\PSL}{\mathrm{PSL}_2}
\newcommand{\Tr}{\mathrm{Tr}}
\newcommand{\E}{\mathrm{e}}
\newcommand{\I}{\mathrm{i}}
\newcommand{\QBinom}[3] {\begin{pmatrix}{#1}\\{#2}\end{pmatrix}_{\kern -4pt{#3}}}
\newcommand{\QInt}[2]{\left(#1\right)_{\kern -1 pt{#2}}}
\newcommand{\QQInt}[2]{\left[#1\right]_{\kern 0 pt{#2}}}

\newcommand{\db}{/\kern -4pt/}

\renewcommand{\leq}{\leqslant}

\renewcommand{\phi}{\varphi}
\renewcommand{\epsilon}{\varepsilon}

\title
[Representations of the skein algebra: punctured surfaces]
{Representations of the\\ Kauffman bracket skein algebra~II: \\punctured surfaces}

\author{Francis Bonahon}
\address {Department
of Mathematics,  University of
Southern California, Los Angeles
CA~90089-2532, U.S.A.}
\email{fbonahon@math.usc.edu}

\author{Helen Wong}
\address{Department
of Mathematics, Carleton College, Northfield MN 55057, U.S.A.}
\email{hwong@carleton.edu}

\thanks{This research was partially supported by grants  DMS-1105402, DMS-1105692, DMS-1406559 from the U.S. National Science Foundation. In addition, the article was extensively rewritten and reorganized while the first author was a Simons  Fellow (grant 301050 from the Simons Foundation) in 2014-15, as well as a Simons Visiting Professor at the Mathematical Sciences Research Institute in Berkeley, California, (NSF grant 09032078000) in the Spring 2015 semester.}

\date{\today}

\begin{document}

\begin{abstract}
In earlier work \cite{BonWon3}, we constructed invariants of irreducible finite-dimen\-sional representations of the Kauffman bracket skein algebra of a surface. We introduce here an inverse construction, which to a set of possible invariants associates an irreducible representation that realizes these invariants. The current article is restricted to surfaces with at least one puncture, a condition that is lifted in subsequent work \cite{BonWon6} relying on this one. A step in the proof is of independent interest, and describes the arithmetic structure of the Thurston intersection form on the space of integer weight systems for a train track. 
\end{abstract}

 \maketitle

This article is a continuation of \cite{BonWon3} and is part of the program described in \cite{BonWon2},   devoted to the analysis and construction of  finite-dimensional representations of the Kauffman bracket skein algebra of a surface.

Let $S$ be an oriented surface of finite topological type without boundary. 
The \emph{Kauffman bracket skein algebra} $\SSS$  depends on a parameter $A=\E^{\pi\I \hbar}\in \C-\{0\}$, and is defined as follows: One first considers the vector space freely generated by  all isotopy classes of framed links in the thickened surface $S \times [0,1]$, and then one takes the quotient of this space by two relations. The first and main relation is the  \emph{skein relation}, which states that 
$$
[K_1] = A^{-1} [K_0] + A [K_\infty]
$$
whenever the three links $K_1$, $K_0$ and $K_\infty\subset S\times [0,1]$ differ only in a little ball where they are as represented in Figure~\ref{fig:SkeinRelation}. The second relation is the \emph{trivial knot relation}, which asserts that 
$$[K\cup O] = -(A^2 +A^{-2})[K]$$
 whenever $O$ is the boundary of a disk $D \subset K \times [0,1]$ disjoint from $K$, and is endowed with a framing transverse to $D$. The algebra multiplication is provided by the operation of superposition, for which the product $[K]\cdot [L]$ is represented by the union $[K'\cup L']$ where $K' \subset S\times [0,\frac 12]$ and $L' \subset S\times [\frac12, 1]$ are respectively obtained by rescaling the framed links $K \subset S\times [0,1]$ and $L' \subset S\times [0, 1]$ in the $[0,1]$ direction. 

\begin{figure}[htbp]

\SetLabels
( .5 * -.4 ) $K_0$ \\
( .1 * -.4 )  $K_1$\\
(  .9*  -.4) $K_\infty$ \\
\endSetLabels
\centerline{\AffixLabels{\includegraphics{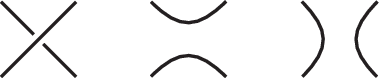}}}
\vskip 15pt
\caption{A Kauffman triple}
\label{fig:SkeinRelation}
\end{figure}

Turaev \cite{Tur}, Bullock-Frohman-Kania-Bartoszy\'nska \cite{BFK1, BFK2} and Przytycki-Sikora \cite{PrzS} showed that the skein algebra $\SSS$ provides a quantization of the \emph{character variety}
$$
\RR = \{ \text{group homomorphisms }r\col  \pi_1(S) \to \SL(\C) \}\db \SL(\C)
$$
where $\SL(\C)$ acts on homomorphisms by conjugation, and where the double bar indicates that the quotient is to be taken in the sense of geometric invariant theory \cite{Mum}. In fact, if one follows the physical tradition that a quantization of a space $X$ replaces the commutative algebra of functions on $X$ by a non-commutative algebra of operators on a Hilbert space, an element of the  quantization should be a \emph{representation} of the skein algebra.

When $A$ is a root of unity, a classical example of a finite-dimensional representation of the skein algebra  $\SSS$  arises from the Witten-Reshetikhin-Turaev topological quantum field theory associated to the fundamental representation of the quantum group $\mathrm U_q (\mathfrak{sl}_2)$ \cite{Witten, ReshTur, BHMV, TuraevBook, BonWon5}. The main purpose of the current article is to provide a wider family of such representations when the surface $S$ has at least one puncture. The case of closed surfaces is considered in the subsequent article \cite{BonWon4}, which  builds on this one.

In \cite{BonWon3}, we identified invariants for irreducible finite-dimensional  representations $\rho \col \SSS \to \End(E)$, in the case where $A^2$ is a primitive $N$--root of unity with $N$ odd. These invariants are a little easier to describe when $A^N=-1$, and most of the current article will be devoted to this case. We indicate in \S \ref{sect:ANplus1} how the other possible case when $A^N=+1$ can be deduced from this one. Because $N$ is odd, the property that $A^2$ is a primitive $N$--root of unity with $A^N=-1$ is equivalent to the property that $A$ is a primitive $N$--root of $-1$. 

When $A^N=-1$, our main invariant is a point of the character variety $\RR$. 
Its definition involves the $n$--th normalized Chebyshev polynomial $T_n(x)$ of the first kind, determined by the trigonometric identity that $2\cos n\theta = T_n(2\cos \theta)$. Equivalently, $\Tr\, M^n = T_n (\Tr\,M)$ for every matrix $M\in \SL(\C)$.

A character $r\in \RR$ associates a trace $\Tr\,r(K)\in \C$ to each closed curve $K$ on the surface $S$. This trace is independent of the homomorphism $\pi_1(S) \to \SL(\C)$  used to represent $r$, or of the representative chosen in the conjugacy class of $\pi_1(S)$ representing $K$. In fact, the character variety $\RR$ is defined in such a way  that two homomorphisms $r \colon \pi_1(S) \to \SL(\C)$ correspond to the same character if and only if they induce the same trace function $K \mapsto \Tr\,r(K)$. 

\begin{thm}[\cite{BonWon3}]
\label{thm:InvariantsExist}
Suppose that $A$ is a primitive $N$--root of $-1$ with $N$ odd, and let $\rho \col \SSS \to \End(E)$ be an irreducible finite-dimensional representation of the Kauffman bracket skein algebra. Let $T_N(x)$ be the $N$-th normalized Chebyshev polynomial  of the first kind. 

\begin{enumerate}
\item  There exists a unique character $r_\rho \in \RR$ such that
$$
T_N \bigl( \rho ([K]) \bigr) =- \bigl( \Tr\, r_\rho(K) \bigr) \Id_E
$$
for every framed knot $K\subset S \times [0,1]$ whose projection to $S$ has no crossing and whose framing is vertical. 

\item Let $P_k$ be a small simple loop going around the $k$--th puncture of $S$, and consider it as a knot in $S\times[0,1]$ with vertical framing. Then there exists  a number $p_k\in \C$ such that
$
\rho\bigl([P_k]\bigr) = p_k \Id_E
$.

\item The number $p_k$ of {\upshape (2)} is related to the character  $r_\rho \in \RR$ of 
 {\upshape (1)} by the property that  $
T_N (p_k) = -\Tr\, r_\rho(P_k)
$.
\end{enumerate}
\end{thm}

The character  $r_\rho \in \RR$ associated to the irreducible representation $\rho \col \SSS \to \End(E)$ by Part~(1) of Theorem~\ref{thm:InvariantsExist} is the \emph{classical shadow} of $\rho$. This Part (1) also admits a generalization to all framed links $K\subset S \times [0,1]$, which involves the element $[K^{T_N}]\in \SSS$ defined by threading the Chebyshev polynomial $T_N$ along all components of $K$; see \cite{BonWon3} for a precise statement. The numbers $p_k$ defined by Part~(2) are the \emph{puncture invariants} of the representation $\rho$. 
Part~(3)   shows that, once the classical shadow $r_\rho$ is known, there are at most $N$ possible values for each of the puncture invariants $p_k$. 

The classical shadow provides one more example of a situation where a quantum object determines one of the classical objects that are being quantized. See also \cite{Le} for a different and more recent approach to  the key results underlying Theorem~\ref{thm:InvariantsExist}.

The main result of this article is the following converse statement. 

\begin{thm}
\label{thm:RealizeInvariantsIntro}
Assume that the surface $S$ has at least one puncture, that its Euler characteristic is negative,   that $A$ is a primitive $N$--root of $-1$ with $N$ odd, and that we are given:
\begin{enumerate}
\item a character  $r\in \RR$ which realizes some ideal triangulation of $S$ in the sense discussed in~{\upshape \S \ref{sect:ShearCoord}};
\item  a number $p_k\in \C$ such that $T_N (p_k) = - \Tr\, r(P_k)$ for each of the punctures of~$S$. 
\end{enumerate} 

Then, there exists an irreducible finite-dimensional representation $\rho \col \SSS \to \End(E)$ whose classical shadow is equal to $r$ and whose puncture invariants are the $p_k$. 
\end{thm}

The requirement that $r$ realizes some ideal representation is fairly mild. It can be shown to be satisfied by all points outside of an algebraic subset of complex codimension $2|\chi(S)| - 1$ in the character variety $\RR$. 

The sequel \cite{BonWon4, BonWon6} to this paper greatly improves Theorem~\ref{thm:RealizeInvariantsIntro}. In particular, it  removes the requirements that $r$ realizes an ideal triangulation, and that $S$ has at least one puncture. It  also shows that the representation provided by our construction is independent of the many choices made during the argument, so that its output is natural, in particular  with respect to the action of the mapping class group $\pi_0 \mathrm{Diff}(S)$ of the surface. The constructions and results of the current article are a key ingredient in the proofs of \cite{BonWon4} and \cite{BonWon6}.

The proof of Theorem~\ref{thm:RealizeInvariantsIntro} uses as a fundamental tool the quantum trace homomorphism $\Tr_\lambda^\omega\col\SSS \to  \TT$, constructed in  \cite{BonWon1}, which embeds the skein algebra in the quantum Teichm\"uller space. The \emph{quantum Teichm\"uller space} is here incarnated as the  Chekhov-Fock  algebra $\TT$ of an ideal triangulation $\lambda$ of the surface, and is a quantization of an object that is closely related to the character variety $\RR$. It is not as natural as the Kauffman bracket skein algebra, but its algebraic structure is very simple. In particular, its representation theory is relatively easy to analyze \cite{BonLiu}. The same holds for a smaller algebra $\ZZ\subset \TT$ containing the image of the quantum trace homomorphism $\Tr_\lambda^\omega$. 
Composing representations of  $\ZZ$ with the homomorphism $\Tr_\lambda^\omega\col\SSS \to \ZZ$ provides an extensive family of representations of the skein algebra $\SSS$, which can then be used to prove Theorem~\ref{thm:RealizeInvariantsIntro}.  

The main technical challenge in this strategy is to compute the classical shadow of the representations of $\SSS$ so obtained, in terms of  the parameters controlling the original representations of $\ZZ$. This is provided by the miraculous cancellations discovered in \cite{BonWon3}. These properties show that the quantum trace homomorphism  $\Tr_\lambda^\omega\col\SSS \to \ZZ$ is well-behaved with respect to the Chebyshev homomorphism $\SSSS \to \SSS$ used to define the classical shadow of a representation of $\SSS$, and with respect to the Frobenius homomorphism $\ZZZ \to \ZZ$ which computes the invariants of representations of $\ZZ$. 

One of the steps in the proof, used to determine the algebraic structure of the algebra $\ZZ$,  may be of interest by itself. This statement describes the  structure of the Thurston intersection form  on the set $\mathcal W(\tau; \Z)$ of integer-valued edge weight systems for a train track $\tau$. The result is well-known for real-valued weights. However, the integer valued case has subtler number-theoretic properties, resulting in the unexpected simultaneous occurrence of blocks 
$\left(
\begin{smallmatrix}
0&1\\-1&0
\end{smallmatrix}
\right)
$
and
$\left(
\begin{smallmatrix}
0&2\\-2&0
\end{smallmatrix}
\right)
$ in the block diagonalization of the Thurston form. See Theorem~\ref{thm:StructureThurston} in the Appendix. Because of the ubiquity of the Thurston intersection form in many geometric problems (for instance the relationship between complex lengths and the shear-bend cocycle $\beta \in \mathcal W(\tau; \C/2 \pi \mathrm i \Z)$ of a pleated surface \cite{Bon97}), this statement is probably of interest beyond the quantum topology scope of the current article.

See the recent preprints \cite{AbdFroh1, AbdFroh2} for another construction of representations of $\SSS$ with a given classical shadow $r\in \RR$, valid for $r$ in a Zariski dense open subset of $\RR$. The construction of \cite{AbdFroh1, AbdFroh2} is simpler, but ours is more explicit. In the few cases where the dimension of the representations of \cite{AbdFroh1, AbdFroh2} can be computed, these dimensions are significantly larger than those arising in the current article. 
 
\section{The Chekhov-Fock algebra and the quantum trace homomorphism}

\subsection{The Chekhov-Fock algebra}
\label{sect:CheFock}

The Chekhov-Fock algebra (introduced in \cite{BonLiu} as a reinterpretation of key insights from  \cite{CheFoc1, CheFoc2, Foc}) is the avatar of the quantum Teich\-m\"uller space associated to an ideal triangulation of the surface $S$. See also \cite{Kash} for a related construction, and \cite{BonLiu, Liu} for more discussion.

If $S$ is obtained from a compact surface $\bar S$ by removing finitely many points $v_1$, $v_2$, \dots, $v_s$, an
 \emph{ideal triangulation} of $S$ is a triangulation $\lambda$ of $\bar S$ whose vertex set is exactly
$\{ v_1,v_2, 
\dots, v_s \}$. The surface $S$ admits an ideal triangulation if and only if it is non-compact and if its Euler characteristic is negative; we will consequently assume these properties satisfied throughout the article. If the surface has genus $g$ and $s$ punctures, an ideal triangulation then has $n=6g+3s-6$  edges and $4g+2s-4$ faces.

Let $e_1$, $e_2$, \dots, $e_n$ denote the edges of $\lambda$. Let $a_i \in \{0,1, 2\}$ be the number of times an end of the edge $e_j$ immediately succeeds an end of $e_i$ when going counterclockwise around a puncture of $S$, and set $\sigma_{ij}=a_{ij}-a_{ji}\in \{-2, -1, 0, 1, 2\}$.  The \emph{Chekhov-Fock algebra} $\TT$ of $\lambda$ is the algebra defined by generators $Z_1^{\pm1}$, $Z_2^{\pm1}$, \dots, $Z_n^{\pm1}$ associated to the edges  $e_1$, $e_2$, \dots, $e_n$ of $\lambda$, and by the relations
$$
Z_iZ_j = \omega^{2\sigma_{ij}} Z_jZ_i.
$$ 

\begin{rem}
The actual Chekhov-Fock algebra $\mathcal T^q(\lambda)$ that is at the basis of the quantum Teichm\"uller space uses the constant $q=\omega^4$ instead of $\omega$. The generators $Z_i$ of $\TT$ appearing here are designed to model square roots of the original generators of $\mathcal T^q(\lambda)$.
\end{rem}

An element of the Chekhov-Fock algebra $\TT$ is a linear combination of monomials $Z_{i_1}^{n_1}Z_{i_2}^{n_2} \dots Z_{i_l}^{n_l}$ in the generators $Z_i$, with $n_1$, $n_2$, \dots, $n_l\in \Z$. Because of the skew-commutativity relation $Z_iZ_j = \omega^{2\sigma_{ij}} Z_jZ_i$, the order of the variables in such a  monomial does matter. It is convenient to use the following symmetrization trick. 
The \emph{Weyl quantum ordering} for  $Z_{i_1}^{n_1}Z_{i_2}^{n_2} \dots Z_{i_l}^{n_l}$ is the monomial
$$
[Z_{i_1}^{n_1}Z_{i_2}^{n_2} \dots Z_{i_l}^{n_l}] = \omega^{-\sum_{u<v} n_un_v\sigma_{i_ui_v}} Z_{i_1}^{n_1}Z_{i_2}^{n_2} \dots Z_{i_l}^{n_l}. 
$$
The formula is specially designed that $[Z_{i_1}^{n_1}Z_{i_2}^{n_2} \dots Z_{i_l}^{n_l}] $ is invariant under any permutation of the $Z_{i_u}^{n_u}$. Note that the algebraic structure of the  Chekhov-Fock algebra $\TT$ depends only on the square $\omega^2$, but that the Weyl quantum ordering depends on the choice of $\omega$.

\subsection{The quantum trace homomorphism}

\begin{thm}[\cite{BonWon1}]
\label{thm:QTrace}
  For $A=\omega^{-2}$, there  exists an injective algebra homomorphism
$$\Tr_\lambda^\omega \col \SSS \to \TT.$$
\end{thm}

The specific  homomorphism $\Tr_\lambda^\omega$ constructed in \cite{BonWon1} is the \emph{quantum trace homomorphism}. It  is uniquely determined by certain properties stated in that article, but we will only need to know that it exists and that it satisfies the properties given in \S \ref{sect:ChebFrob} below.

\subsection{The Chebyshev and Frobenius homomorphisms}
\label{sect:ChebFrob}

We now assume that $A$ is a primitive $N$--root of $-1$  with $N$  odd.  Recall that $T_N$ denotes the $N$--th normalized Chebyshev polynomial, defined by the property that $\cos N\theta = \frac12 T_N(2\cos \theta)$ for every $\theta$. 

\begin{thm}[\cite{BonWon3}]
\label{thm:ChebSkeinRelation}
When $A$ is a primitive $N$--root of $-1$  with $N$  odd, there is a unique algebra homomorphism $\mathbf T^A \col \SSSS \to \SSS$ such that
$$
\mathbf T^A \bigl(  [K] \bigr) = T_N  \bigl(  [K] \bigr) 
$$
for every framed knot $K\subset S \times [0,1]$ whose projection to $S$ has no crossing and whose framing is vertical. In addition, the image of $\mathbf T^A$ is central in $\SSS$. \qed
\end{thm}

The homomorphism  $\mathbf T^A $ provided by Proposition~\ref{thm:ChebSkeinRelation} is the \emph{Chebyshev homomorphism}. It is a key ingredient in the definition of the invariants of Theorem~\ref{thm:InvariantsExist}. 

There is an analogous and much simpler homomorphism at the level of the Chekhov-Fock algebra, namely the following \emph{Frobenius homomorphism}. 

\begin{prop}
\label{prop:Frobenius}
If  $\iota = \omega^{N^2}$, there  is an algebra homomorphism
$$
\mathbf F^\omega \col \TTT   \to \TT$$ 
which maps each generator $Z_i\in \TTT$ to $Z_i^N\in \TT$, where in the first instance $Z_i\in \TTT$ denotes the generator associated to the $i$--th edge $e_i$ of $\lambda$, whereas the second time $Z_i\in \TT$ denotes the generator of $\TT$ associated to the same edge~$e_i$. \qed
\end{prop}

Note that $\iota^2 = \omega^{2N^2} = A^{-N^2} =(-1)^N =-1$, so that $\iota = \pm \mathrm i$. 

The  following compatibility statement, which connects the Chebyshev homomorphism to the Frobenius homomorphism through appropriate  quantum trace homomorphisms, is fundamental for our arguments. This result encapsulates the miraculous cancellations of \cite{BonWon3}.

\begin{thm}[\cite{BonWon3}]
\label{thm:ChebyQTracesFrob}
The diagram 
$$
\xymatrix{
\SSS
 \ar[r]^{\Tr_{\lambda}^\omega} 
 & \TT\\
\SSSS
 \ar[r]^{\Tr_{\lambda}^\iota}
 \ar[u]^{\mathbf T^A}
 & \TTT
 \ar[u]_{\mathbf F^\omega}}
$$
is commutative. Namely, for every skein $[K]\in \SSSS$, the quantum trace $\Tr_\lambda^\omega\bigl(\mathbf T^A\bigl( [K]\bigr) \bigr)$ of $ \mathbf T^A\bigl([K]\bigr)$ is obtained from the classical trace polynomial $\Tr_\lambda^\iota \bigl([K]\bigr)$ by replacing each generator $Z_i \in \TTT$ by $Z_i^N\in \TT$. \qed
\end{thm}

\section{The balanced Chekhov-Fock  algebra}

\subsection{The balanced Chekhov-Fock  algebra}
The quantum trace homomorphism  $\Tr^\omega_\lambda$ of Theorem~\ref{thm:QTrace} (and \cite{BonWon1}) is far from being surjective. Indeed, for a skein $[K]\in \SSS$ represented by a framed link $K \subset S \times [0,1]$,  the exponents of the monomials $Z_1^{k_1}Z_2^{k_2} \dots Z_n^{k_n}$ appearing in the expression of $\Tr^\omega_\lambda\bigl([K]\bigr)$  are \emph{balanced}, in the sense that they satisfy the following {parity condition}: for every triangle $T_j$ of the ideal triangulation $\lambda$, the sum $k_{i_1}+ k_{i_2}+ k_{i_3}$ of the exponents of the generators $Z_{i_1}$, $Z_{i_2}$, $Z_{i_3}$ associated to the sides of $T_j$ is even. 

Let $\ZZ$ denote the sub-algebra of $\TT$ generated by all monomials satisfying this exponent parity condition. By definition, $\ZZ$ is the \emph{balanced Chekhov-Fock algebra} of the ideal triangulation $\lambda$. It is designed so that the quantum trace homomorphism restricts to a homomorphism $\Tr^\omega_\lambda \colon \SSS \to \ZZ$. 

\begin{figure}[htbp]

\SetLabels
( .3* .25) $T_j $ \\
( .5*.57 ) $\tau_\lambda $ \\
\endSetLabels
\centerline{\AffixLabels{ \includegraphics[width=4cm]{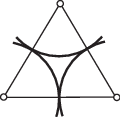}}}
 \caption{}
\label{fig:TriangleTrainTrack}
\end{figure}

To keep track of the exponent parity condition defining the monomials of $\ZZ$, it is convenient to consider a train track $\tau_\lambda$ which, on each triangle $T_j$ of the ideal triangulation $\lambda$, looks as in Figure~\ref{fig:TriangleTrainTrack}. In particular, $\tau_\lambda$ has one switch for each edge of $\lambda$, and three edges for each triangle of $\lambda$. Let $\mathcal W(\tau_\lambda; \Z)$ be the set of integer edge weight systems $\alpha$ for $\tau_\lambda$, assigning a number $\alpha(e) \in \Z$ to each edge $e$ of $\tau_\lambda$ in such a way that, at each switch, the weights of the edges incoming on one side add up to the sum of the weights of the edges outgoing on the other side. There is a natural  map $\mathcal W(\tau_\lambda; \Z) \to \Z^n$ which, given an edge weight system,  associates to each of the $n$ switches  of $\tau_\lambda$ the sum of the weights of the edges incoming on any side of the switch. Then, an element $(k_1, k_2, \dots, k_n) \in \Z^n$ satisfies the above parity condition if and only if  it is in the image of this map. Also, the map $\mathcal W(\tau_\lambda; \Z) \to \Z^n$ is easily seen to be injective. Since the image of this map has finite index, it follows that $\mathcal W(\tau_\lambda; \Z)$ is isomorphic to $\Z^n$ as an abelian group. 

This enables us to give a different description of $\ZZ$. For a weight system $\alpha \in \mathcal W(\tau_\lambda; \Z)$, assigning a weight $\alpha_i\in \Z$ to the $i$--th edge $e_i$ of $\lambda$ (= the $i$--th switch of $\tau_\lambda$),  define
$$
Z_\alpha = [ Z_1^{\alpha_1}Z_2^{\alpha_2} \dots Z_n^{\alpha_n}] \in \ZZ
$$
where the bracket $[\enspace]$ denotes the Weyl quantum ordering defined in \S \ref{sect:CheFock}. 

The above discussion proves the following fact.

\begin{lem}
\label{lem:TrainTrackSqRootCheFock}
As $\alpha \in \mathcal W(\tau_\lambda; \Z)$ ranges over all weight systems for the train track $\tau_\lambda$, the associated  monomials $Z_\alpha$ form a basis for the vector space $\ZZ$. \qed
\end{lem}

We can elaborate a little on the structure of the group $ \mathcal W(\tau_\lambda; \Z)$. By definition of the parity condition, $ \mathcal W(\tau_\lambda; \Z)\subset \Z^n $ contains the subset $(2\Z)^n$ consisting of all switch weight systems $(\alpha_1, \alpha_2, \dots, \alpha_n) \in \Z^n$ where the $\alpha_i$ are even. Also, given $\alpha \in \mathcal W(\tau_\lambda; \Z)$, we can define a chain with coefficients in $\Z_2$ by endowing each edge $e$ of the train track $\tau_\lambda$ with the modulo 2 reduction of the weight $\alpha(e) \in \Z$. The switch relations guarantee that this chain is closed, and this defines a natural homomorphism $\mathcal W(\tau_\lambda; \Z) \to H_1(S; \Z_2)$. 

\begin{lem}
\label{lem:StructureWeightSystems}
The  inclusion map and homomorphism above define an exact sequence
$$
0 \to (2\Z)^n \to \mathcal W(\tau_\lambda; \Z) \to H_1(S; \Z_2) \to 0.
$$
\end{lem}
\begin{proof}
The homomorphism $ \mathcal W(\tau_\lambda; \Z) \to H_1(S; \Z_2)$ can also be expressed in terms of the dual graph $\Gamma_\lambda$ of the triangulation $\lambda$. Indeed, the class $[\alpha ] \in H_1(S; \Z_2)$ induced by  $\alpha \in \mathcal W(\tau_\lambda; \Z)$ is also realized by endowing each edge $f_i$ of $\Gamma_\lambda$ with the modulo 2 reduction of the switch weight $\alpha_i$ associated by $\alpha$ to the  edge $e_i$ of $\lambda$ that is dual to $f_i$; the parity condition guarantees that this chain is really closed.  The result then immediately follows from the definitions, and from the isomorphism $H_1(\Gamma_\lambda ; \Z_2)\cong H_1(S; \Z_2)$ coming from the fact that the surface $S$ deformation retracts to the dual graph $\Gamma_\lambda$. 
\end{proof}

Note that the exact sequence of Lemma~\ref{lem:StructureWeightSystems} admits no partial splitting. 

\subsection{The algebraic structure of the balanced Chekhov-Fock  algebra}
\label{subsect:AlgStructCheFock}

We first describe the multiplicative structure of the balanced Chekhov-Fock algebra  $\ZZ$ in the context of Lemma~\ref{lem:TrainTrackSqRootCheFock}. 

The space $ \mathcal W(\tau_\lambda; \Z)$ carries a very natural antisymmetric bilinear form
$$
\Omega \col  \mathcal W(\tau_\lambda; \Z) \times  \mathcal W(\tau_\lambda; \Z) \to \Z,
$$
the \emph{Thurston intersection form} defined by the property that, for $\alpha$, $\beta\in  \mathcal W(\tau_\lambda; \Z)$,
$$
\Omega(\alpha, \beta) =   {\textstyle\frac12}\kern -10pt \sum_{e \text{ right of }e'} \kern -10pt  \bigl( \alpha(e)\beta(e') - \alpha(e')\beta(e) \bigr)
$$
where the sum is over all pairs $(e,e')$ of edges of $\tau_\lambda$ such that $e$ and $e'$ come out of the same side of some switch of $\tau_\lambda$, with $e$ to the right of $e'$.  See Lemma~\ref{lem:ThurstonHalfIntersection} in the  Appendix for a more conceptual interpretation of $\Omega$,  and for a proof that  $\Omega(\alpha, \beta)$ is really an integer.

\begin{lem}
\label{lem:SqRootCheFockThurston}
For every $\alpha$, $\beta\in  \mathcal W(\tau_\lambda; \Z)$,
$$
Z_\alpha Z_\beta = \omega^{2\Omega(\alpha, \beta)} Z_{\alpha+\beta}.
$$
In particular, $Z_\alpha Z_\beta =  \omega^{4\Omega(\alpha, \beta)} Z_\beta Z_\alpha$.
\end{lem}
\begin{proof}
The second statement, that $Z_\alpha Z_\beta =  \omega^{4\Omega(\alpha, \beta)} Z_\beta Z_\alpha$, is a simple computation. After observing that this property holds for any $\omega$ (not just roots of unity), the  first statement, that $Z_\alpha Z_\beta = \omega^{2\Omega(\alpha, \beta)} Z_{\alpha+\beta}$, then follows by definition of the Weyl quantum ordering.
\end{proof}

This is particularly simple if we replace $\omega$ by $\iota=\omega^{N^2}$, with the assumption that $A^{2N}=1$ so that $\iota^{4} = \omega^{4N^2} =A^{-2N^2}=1$. 

\begin{cor}
\label{cor:SquareRootChefockCommutativeIota}
If $\iota^4=1$, the algebra $\ZZZ$ is commutative. \qed
\end{cor}

In general, the key to understanding the algebraic structure of $\ZZ$ is Lemma~\ref{lem:StructureThurston} below. 

For $k=1$, \dots, $s$, the $k$--th puncture of $S$ specifies an element $\eta_k \in  \mathcal W(\tau_\lambda; \Z)$ defined as follows: For every edge $e$ of $\tau_\lambda$, the edge weight  $\eta_k(e)\in \{0,1,2\}$ is the number of sides of $e$ that are adjacent to the same component of $S-\tau_\lambda$ as this puncture.  

Recall that the surface $S$ has genus $g$ and $s$ punctures. 

\begin{lem}
\label{lem:StructureThurston}

The lattice $ \mathcal W(\tau_\lambda; \Z) \cong \Z^n$ admits a basis in which the matrix of the Thurston intersection form $\Omega$ is block diagonal with $g$ blocks 
$\begin{pmatrix}
0&1\\-1&0
\end{pmatrix}$, 
$2g+s-3$ blocks 
$\begin{pmatrix}
0&2\\-2&0
\end{pmatrix}$
and $s$ blocks
$\begin{pmatrix}
0
\end{pmatrix}$. In addition, the kernel of $\Omega$ is freely generated by the elements $\eta_1$, $\eta_2$, \dots, $\eta_s \in  \mathcal W(\tau_\lambda; \Z) $ associated to the punctures of $S$ as above. 
\end{lem}

\begin{proof}
This is a special case of a result given by Theorem~\ref{thm:StructureThurston} in the Appendix, which determines the algebraic structure of the Thurston intersection form for a general train track $\tau$. When applying this result to the train track $\tau_\lambda$, the numbers $h$, $n_{\mathrm{even}}$ and $n_{\mathrm{odd}}$ of Theorem~\ref{thm:StructureThurston} are respectively equal to the genus $g$ of the surface $S$,  to the number $s$ of punctures of $S$, and  to the number $4g+2s-4$ of triangles of the ideal triangulation $\lambda$. 
\end{proof}

The combination of Lemmas~\ref{lem:TrainTrackSqRootCheFock}, \ref{lem:SqRootCheFockThurston} and \ref{lem:StructureThurston}  now provides the complete algebraic structure of the  balanced Chekhov-Fock algebra $\ZZ$. Let $\mathcal W^q$ denote the algebra, known as the quantum 2--torus,  defined by generators $X^{\pm1}$, $Y^{\pm1}$ and by the relation $XY = q YX$. 

\begin{cor}
\label{cor:StructureBalancedCheFock}
For $q=\omega^4$, the balanced Chekhov-Fock algebra $\ZZ$ is isomorphic to
$$
\mathcal W^q_1 \otimes \mathcal W^q_2 \otimes \dots \otimes \mathcal W^q_{g} \otimes 
\mathcal W^{q^2}_{g+1} \otimes \mathcal W^{q^2}_{g+2} \otimes \dots \otimes \mathcal W^{q^2}_{3g+s-3}  \otimes \C[H_1] \otimes  \C[H_2] \otimes \dots \otimes \C[H_s]
$$
where each $\mathcal W^q_i$ is a copy of the quantum $2$--torus $\mathcal W^q$, each $\mathcal W^{q^2}_j$ is a copy of  $\mathcal W^{q^2}$, and each $ \C[H_k] $ is a polynomial algebra in the variable $H_k$. 

In addition, the $s$ central generators $H_k=Z_{\eta_k}$ are associated to the punctures of $S$ as in Lemma~{\upshape\ref{lem:StructureThurston}}. \qed
\end{cor}

\subsection{Representations of the balanced Chekhov-Fock algebra}
\label{sect:RepBalancedCF}

The algebraic structure of the balanced  Chekhov-Fock algebra $\ZZ$ determined in Corollary~\ref{cor:StructureBalancedCheFock} is relatively simple. This makes it easy to classify its irreducible finite-dimensional representations. 

As usual, we assume that $A=\omega^{-2}$ is a primitive $N$--root of $-1$, with $N$ odd. 

\begin{prop}
\label{prop:InvariantsRepsCheFock}
Let $\mu\col \ZZ \to \End(E)$ be an irreducible finite-dimensional representation of $\ZZ$. There exists a   map $\zeta_\mu\col \mathcal W(\tau_\lambda; \Z) \to \C^*$ and  numbers $h_k \in \C^*$, with $k=1$, \dots, $s$, associated to the punctures of the surface $S$ such that:
\begin{enumerate}
\item $\mu(Z_\alpha^N) = \zeta_\mu(\alpha) \,\Id_E$ for every edge weight system $\alpha \in \mathcal W(\tau_\lambda; \Z)$ with associated monomial $Z_\alpha \in \ZZ$; 
\item $\zeta_\mu(\alpha + \beta) = (-1)^{\Omega(\alpha, \beta)} \zeta_\mu(\alpha) \zeta_\mu(\beta)$ for every $\alpha$, $\beta \in \mathcal W(\tau_\lambda; \Z)$, where $\Omega$ is the Thurston intersection form;
\item $\mu(H_k) = h_k\,\Id_E$ for the central element $H_k =Z_{\eta_k}\in \ZZ$ associated to the $k$--th puncture of $S$ as in Corollary~{\upshape\ref{cor:StructureBalancedCheFock}}; 
\item $\zeta_\mu(\eta_k) = h_k^N$ for the weight system $\eta_k\in \mathcal W(\tau_\lambda; \Z)$ associated to the $k$--th puncture of $S$ as in Lemma~{\upshape\ref{lem:StructureThurston}}.
\end{enumerate}
\end{prop}

\begin{proof}
For every $\alpha \in \mathcal W(\tau_\lambda; \Z)$, Lemma~\ref{lem:SqRootCheFockThurston} shows that the element $Z_\alpha^N = Z_{N\alpha}$ is central in $\ZZ$. In particular, if $\mu\col \ZZ \to \End(E)$ is an irreducible finite-dimensional representation of $\ZZ$, there is a number $\zeta_\mu(\alpha)\in \C^*$ such that
$
\mu(Z_\alpha^N) = \zeta_\mu(\alpha) \,\Id_E
$.
In addition, Lemma~\ref{lem:SqRootCheFockThurston} shows that $Z_\alpha^N Z_\beta^N = \omega^{2N^2\Omega(\alpha, \beta)}Z_{\alpha+\beta}^N= (-1)^{\Omega(\alpha, \beta)}Z_{\alpha+\beta}^N$, so that the  map $\zeta_\mu\col \mathcal W(\tau_\lambda; \Z) \to \C^*$ satisfies Property~(2). 

Similarly, Corollary~\ref{cor:StructureBalancedCheFock} shows that each $H_k$ is central in $\ZZ$, so that $\mu(H_k) = h_k \, \Id$ for some $h_k \in \C^*$. Then $h_k^N \,\Id_E = \mu(H_k^N) = \mu (Z_{\eta_k}^N) = \zeta_\mu(\eta_k)\, \Id_E$ since $H_k = Z_{\eta_k}$, so that $ \zeta_\mu(\eta_k) = h_k^N$. 
\end{proof}

A map $\zeta\col \mathcal W(\tau_\lambda; \Z) \to \C^*$ satisfying Condition~(2) of Proposition~\ref{prop:InvariantsRepsCheFock} is a \emph{twisted homomorphism} twisted by the Thurston form  $\Omega$, or more precisely twisted by the symmetric map $(\alpha, \beta) \mapsto (-1)^{\Omega(\alpha, \beta)}$. This notion will probably look less intimidating once one realizes that   a twisted homomorphism is completely determined  by the assignment of a non-zero complex number to each of the $n$ generators of the group $ \mathcal W(\tau_\lambda; \Z) \cong \Z^n$.

\begin{prop}
\label{prop:RepsSquareRootCheFock}
Suppose that we are given a twisted homomorphism  $\zeta\col \mathcal W(\tau_\lambda; \Z) \to \C^*$ twisted by the Thurston form $\Omega$  and, for each of the punctures of $S$, a number $h_k\in \C^*$ such that $h_k^N = \zeta(\eta_k)$. Then, up to isomorphism, there exists a unique  irreducible finite-dimensional representation $\mu\col \ZZ \to \End(E)$ such that
\begin{enumerate}
\item $\zeta_\mu = \zeta$, namely $\mu(Z_\alpha^N) = \zeta(\alpha) \,\Id_E$ for every $\alpha \in \mathcal W(\tau_\lambda; \Z)$;
\item $\mu(H_k) = h_k\,\Id_E$ for $k=1$, \dots, $s$.
\end{enumerate}

In addition, for such a representation, the vector space $E$ has dimension $N^{3g+s-3}$. 
\end{prop}

\begin{proof}
Using elementary linear algebra, this is an immediate consequence of Corollary~\ref{cor:StructureBalancedCheFock}.
More precisely, consider the isomorphism
$$
\ZZ \cong\mathcal W^q_1 \otimes \dots \otimes \mathcal W^q_{g} \otimes 
\mathcal W^{q^2}_{g+1} \otimes \dots \otimes \mathcal W^{q^2}_{3g+s-3}  \otimes \C[H_1] \otimes  \dots \otimes \C[H_s]
$$
provided by Corollary~\ref{cor:StructureBalancedCheFock}. 

For $1\leq i\leq 3g+s-3$, let $X_i^{\pm1}$ and $Y_i^{\pm1}$ denote the generators of $\mathcal W^q_i$ or $\mathcal W^{q^2}_i$  (satisfying the relation $X_iY_i = qY_iX_i$ if $1\leq i\leq g$ and $X_iY_i = q^2Y_iX_i$ if $g< i\leq 3g+s-3$). The proof of Corollary~\ref{cor:StructureBalancedCheFock} shows that  these generators are of the form $X_i=Z_{\alpha_i}$, $Y_i = Z_{\beta_i}$ and $H_k = Z_{\eta_k}$ for some edge weight systems $\alpha_i$, $\beta_i$, $\eta_k  \in \mathcal W(\tau_\lambda; \Z)$. In addition, the $\alpha_i$, $\beta_i$ and  $\eta_k$ form a basis for $ \mathcal W(\tau_\lambda; \Z)\cong \Z^n$.

Because $N$ is odd, $q=\omega^4$ and $q^2$ are both primitive $N$--root of unity.  Arbitrarily pick $N$--roots $\zeta(\alpha_i)^{\frac1N}$ and $\zeta(\beta_i)^{\frac1N}$, and  define $\mu_i\col \mathcal W^q_i \to \End(E_i)$ by the property that, if $v_1$, $v_2$, \dots, $v_N$ form a basis for $E_i \cong \C^N$,
\begin{itemize}
\item[] $\mu_i(X_i) (v_j) = -\zeta(\alpha_i)^{\frac1N} q^j v_j$ and $\mu_i(Y_i) (v_j) = \zeta(\beta_i)^{\frac1N} v_{j+1}$
if $1\leq i \leq g$, and
\item[] 
$\mu_i(X_i) (v_j) = \zeta(\alpha_i)^{\frac1N} q^{2j} v_j$  and $ \mu_i(Y_i) (v_j) = \zeta(\beta_i)^{\frac1N} v_{j+1}$
if $g< i\leq 3g+s-3$. 
\end{itemize}

Then, for $E = E_1 \otimes E_2\otimes \dots \otimes E_{3g+s-3}$, define $\mu \col \ZZ \to \End(E)$ by the property that $\mu$ coincides with $\mu_1 \otimes \mu_2\otimes \dots \otimes \mu_{3g+s-3}$ on $ \mathcal W^q_1 \otimes \mathcal W^q_2 \otimes \dots \otimes \mathcal W^q_{3g+s-3} $, and  $\mu(H_k) = h_k\,\Id_E$ for every $k=1$, \dots, $s$. 

It is immediate that $\mu$ satisfies the required properties. The fact that $\mu$ is irreducible, and that every irreducible representation is isomorphic to $\mu$, is easily proved by elementary linear algebra; see for instance \cite[\S 4]{BonLiu} for details. 
\end{proof}

\section{Pleated surfaces and homomorphisms to $\SL(\C)$}
\label{sect:ShearCoord}

Let us consider the special case of Proposition~\ref{prop:RepsSquareRootCheFock} when $N=1$. In particular, $A=-1$ and $\iota = \omega = \pm \mathrm i$. Since the Chebyshev polynomial $T_1(x)$ is equal to $x$, the choice of puncture invariants $h_k$ is irrelevant and  Proposition~\ref{prop:RepsSquareRootCheFock} associates to any twisted homomorphism $\zeta\col \mathcal W(\tau_\lambda; \Z) \to \C^*$ a representation $\mu_\zeta\col \ZZZ \to \End(\C)$. By composition with the quantum trace homomorphism $\Tr_\lambda^\iota \col \SSSS \to \ZZZ $ of  Theorem~\ref{thm:QTrace}, we now have a homomorphism
$$
\rho_\zeta = \mu_\zeta \circ \Tr_\lambda^\iota \col \SSSS \to \End(\C) = \C.
$$
We can then apply the case $N=1$ of Theorem~\ref{thm:InvariantsExist} (which actually is an observation of Doug Bullock, Charlie Frohman, Joanna Kania-Bartoszy\'nska, Jozef Przytycki and Adam Sikora \cite{Bull1, Bull2, BFK1, BFK2, PrzS} and plays a crucial r\^ole in the proof of Theorem~\ref{thm:InvariantsExist}  in its full generality). It provides a character $r_\zeta\in\RR$ such that
$$
\rho_\zeta \bigl([K]\bigr) = - \Tr\, r_\zeta(K)
$$
for every framed knot $K \subset S\times[0,1]$. The property is valid for all knots, not just those whose projection to $S$ has no double point  \cite{Bull1, Bull2, BFK1, BFK2, PrzS}.

It is natural to ask which elements of $\RR$ are obtained in this way. The answer involves the following geometric definition. 

Let $\widetilde S$ be the universal cover of $S$, and let $\widetilde\lambda$ be the ideal triangulation of $\widetilde S$ obtained by lifting the edges and faces of $\lambda$. Identify $\PSL(\C)$ to the isometry group of the hyperbolic 3--space $\HH^3$. A \emph{pleated surface} with \emph{pleating locus} $\lambda$ is the data $(\widetilde f, \bar r)$ of a map $\widetilde f \col \widetilde S \to \HH^3$ and a group homomorphism $\bar r\col \pi_1(S) \to \PSL(\C)$ such that:
\begin{enumerate}
\item $\widetilde f$ homeomorphically sends each edge of $\widetilde \lambda$ to a complete geodesic of the hyperbolic space $\HH^3$, and every face of $\widetilde \lambda$ to a totally geodesic ideal triangle of $\HH^3$, with vertices on the sphere at infinity $\partial_\infty \HH^3$;
\item $\widetilde f$ is $\bar r$--equivariant, in the sense that $\widetilde f( \gamma \widetilde x) = \bar r(\gamma) \bigl( \widetilde f(\widetilde x) \bigr)$ for every $\gamma \in \pi_1(S)$ and every $\widetilde x \in \widetilde S$. 
\end{enumerate}

Following the terminology introduced in \cite{Thu}, we say that the group homomorphism $\bar r\col \pi_1(S) \to \PSL(\C)$ \emph{realizes} the ideal triangulation $\lambda$ if there exists a pleated surface $(\widetilde f, \bar r)$ with pleating locus $\lambda$. By extension, a point in the character variety $\mathcal R_{\PSL(\C)}(S)$ \emph{realizes} $\lambda$ if it can be represented by a homomorphism $\bar r\col \pi_1(S) \to \PSL(\C)$ realizing $\lambda$. Finally,  a  character in $\RR$  \emph{realizes} $\lambda$ if it is sent to a point of $\mathcal R_{\PSL(\C)}(S)$ realizing $\lambda$ by  the canonical projection $\RR\to \mathcal R_{\PSL(\C)}(S)$ 

We are now ready to state the result promised. At the beginning of this section, we associated a character $r_\zeta \in \RR$ to each  twisted homomorphism $\zeta\col \mathcal W(\tau_\lambda; \Z) \to \C^*$. 

\begin{prop}
\label{prop:RealizeIdealTriang}
A  character  $r\in\RR$ is associated to a twisted homomorphism $\zeta\col \mathcal W(\tau_\lambda; \Z) \to \C^*$ as above if and only it realizes the ideal triangulation $\lambda$. 
\end{prop}

\begin{proof} Suppose that $r\in\RR$ realizes the ideal triangulation $\lambda$. By definition, there exists a pleated surface $(\widetilde f, \bar r)$ with pleating locus $\lambda$, where the homomorphism $\bar r\col \pi_1(S) \to \PSL(\C)$ represents the image of $r$ under the projection $\RR\to \mathcal R_{\PSL(\C)}(S)$. 

The pleated surface $(\widetilde f, \bar r)$ determines, for each edge $\widetilde e_i$ of the ideal triangulation $\widetilde \lambda$ of $\widetilde S$, a complex weight $\widetilde x_i\in \C^*$ defined as follows: If $\widetilde Q_i \subset \widetilde S$ is the quadrilateral formed by the two faces of $\widetilde \lambda$ meeting along the edge $\widetilde e_i$, then $-\widetilde x_i$ is the cross-ratio of the 4 vertices of $\widetilde f(\widetilde Q_i)$ in the sphere at infinity $\C \cup \{ \infty \}$ of $\HH^3$. These edge weights $\widetilde x_i$ are equivariant under the action of $\pi_1(S)$, and therefore descend to a system of weights $x_i$ for the edges $e_i$ of $\lambda$.  The edge weights $x_i \in \C^*$ are the \emph{shear-bend parameters} of the pleated surface $(\widetilde f, \bar r)$. 

Choose square roots $z_i = \sqrt{x_i}$. Then, for every closed curve $K$ in $S$, there is an explicit formula that expresses the trace $\Tr\, \bar r(K)$ as a Laurent polynomial in the $z_i$; see for instance \cite[\S\S 1.3--1.4]{BonWon1}. Note that there necessarily is a sign ambiguity in this formula, as the trace of an element of $\PSL(\C)$ is only defined up to sign. Another sign ambiguity occurs in the choice of the square roots $z_i = \sqrt{x_i}$.

We will use these edge weights $z_i \in \C^*$  to construct representations of $\ZZZ$ and $\SSSS$. Recall that a twisted homomorphism $\zeta\col \mathcal W(\tau_\lambda; \Z) \to \C^*$ is equivalent to the data of its value on a set of generators of  $\mathcal W(\tau_\lambda; \Z) \cong Z^n$. We can therefore find such a twisted homomorphism so that
$$\zeta( \alpha ) = \pm  z_1^{\alpha_1}z_2^{\alpha_2}\dots z_n^{\alpha_n}$$
for every edge weight system $\alpha \in  \mathcal W(\tau_\lambda; \Z)$ assigning weight $\alpha_i \in \Z$ to the edge $e_i$ of $\lambda$. The $\pm$ signs are here required by the twisting. 
In addition, a simple manipulation of the formula for the Thurston intersection form (or a use of Lemma~\ref{lem:SqRootCheFockThurston}) show that $\Omega(\alpha, \beta)$ is even  whenever $\alpha \in  (2\Z)^n \subset \mathcal W(\tau_\lambda; \Z)$ assigns even weights $\alpha_i \in \Z$ to all edges of $\lambda$; in particular, there is no twisting on $(2\Z)^n \subset \mathcal W(\tau_\lambda; \Z)$. Using Lemma~\ref{lem:StructureWeightSystems}, we can therefore arrange that
$$\zeta( \alpha ) = +  z_1^{\alpha_1}z_2^{\alpha_2}\dots z_n^{\alpha_n}$$
for every $\alpha \in  (2\Z)^n $. 

Note that there are several possible choices for $\zeta$, coming from the signs $\pm$. In fact, Lemma~\ref{lem:StructureWeightSystems} shows that there are exactly $2^d$ possibilities for $\zeta$, where $d$ is the dimension of $H_1(S;\Z_2)$. We will later adjust the choice of $\zeta$ so that it fits our purposes. 

Let $ \mu_\zeta \colon \ZZZ \to \End(\C)= \C$ be the representation of $\ZZZ$ associated by Proposition~\ref{prop:RepsSquareRootCheFock} to the twisted homomorphism $\zeta$. Namely, $ \mu(Z_\alpha)= \zeta(\alpha)$ for every $\alpha \in  \mathcal W(\tau_\lambda; \Z)$

The definition of  the quantum trace homomorphism $\Tr_\lambda^\iota \col \SSSS \to \ZZZ $ in \cite{BonWon1} was specially designed to copy the formula expressing the trace $\Tr\, \bar r(K)$ as a Laurent polynomial in the square roots $z_i=\sqrt{x_i}$ of the shear-bend parameters of the pleated surface  $(\widetilde f, \bar r)$. In particular, because of the key property that $\mu_\zeta(Z_i^2)=+z_i^2$, 
$$
\mu_\zeta \circ \Tr_\lambda^\iota \bigl([K] \bigr) = \pm \Tr \,  r(K)
$$
for every framed knot $K\subset S \times [0,1]$, where the sign $\pm$ depends on $K$ and on the choice of the square roots  $z_i=\sqrt{x_i}$; see the discussion in \cite[\S\S 1.3--1.4]{BonWon1}.

As discussed at the beginning of this section, the homomorphism $ \mu_\zeta \circ \Tr_\lambda^\iota \col \SSSS \to  \C$ also defines a character $r_\zeta \in \RR$ such that 
$$
\mu_\zeta \circ \Tr_\lambda^\iota \bigl([K]\bigr) = - \Tr\, r_\zeta(K)
$$
for every framed knot $K \subset S\times[0,1]$. As a consequence, $ \Tr\, r_\zeta(K) = \pm  \Tr\, r(K)$ for every knot $K$. 

At this point, there is no reason for the two characters $r$ and $r_\zeta\in \RR$ to coincide. However, by construction, they project to the same $\PSL(\C)$--valued character in $\mathcal R_{\PSL(\C)}(S)$. Their difference can therefore be encoded by  a cohomology class $\epsilon \in H^1(S; \Z_2)$, such that 
$$
\Tr\, r(K) =(-1)^{ \epsilon (K)} \Tr\, r_\zeta(K)
$$
for every knot $K\subset S \times [0,1]$. 

Each edge $e_i$ of the ideal triangulation $\lambda$ is Poincar\'e dual to a cohomology class $\epsilon_i \in H^1(S; \Z_2)$. Replacing the square root $z_i = \sqrt{x_i}$ by the other square root $-z_i$ has the effect of replacing $r_\zeta$ with $\epsilon_i r_\zeta$. Since the $\epsilon_i $ generate $ H^1(S; \Z_2)$, we can therefore adjust the choice of the square roots $z_i = \sqrt{x_i}$ so that  the characters $r$ and $r_\zeta\in \RR$ are now equal. 

This proves that, if the character $r\in \RR$ realizes the ideal triangulation $\lambda$, there exists a twisted homomorphism $\zeta\col \mathcal W(\tau_\lambda; \Z) \to \C^*$ whose associated character $r_\zeta \in \RR$ is equal to $r$. 

Conversely, suppose that $r=r_\zeta \in \RR$ is associated to a twisted homomorphism $\zeta\col \mathcal W(\tau_\lambda; \Z) \to \C^*$ as above. More precisely, consider the corresponding representation $\mu_\zeta\col \ZZZ \to \End(\C)= \C$, defined by the property that $\mu_\zeta(Z_\alpha)=\zeta(\alpha)$ for every $\alpha \in \mathcal W(\tau_\lambda; \Z) $. Then 
$$
 \mu_\zeta \circ \Tr_\lambda^\iota  \bigl([K]\bigr) = - \Tr\, r(K)
$$
for every framed knot $K \subset S\times[0,1]$

The generator $Z_i \in \TTT$ associated to the edge $e_i$ of $\lambda$ does not satisfy the exponent parity condition defining the balanced Chekhov-Fock algebra $\ZZZ$, but its square does. We can therefore consider $x_i = \mu_\zeta(Z_i^2) \in \C$, which is different from 0 since $Z_i^2$ is invertible. 

We can then construct a pleated surface $(\widetilde f, \bar r)$ whose pleating locus is equal to $\lambda$ and whose shear-bend parameters are equal to the edge weights $x_i$.  In particular, this pleated surface is equivariant with respect to a homomorphism $\bar r\colon \pi_1(S) \to \PSL(\C)$, which defines a character $\bar r \in \mathcal R_{\PSL(\C)}(S)$. 

By our discussion of the geometric interpretation of the trace homomorphism $\Tr_\lambda^\iota  $, the character $\bar r \in \mathcal R_{\PSL(\C)}(S)$ is the projection of $r\in \RR$. In particular, $r$ realizes the ideal triangulation~$\lambda$. 

This concludes the proof of Proposition~\ref{prop:RealizeIdealTriang}. 
\end{proof}

\section{Representations of the skein algebra}

We are now ready to prove Theorem~\ref{thm:RealizeInvariantsIntro}. 

We begin with an elementary lemma about the Chebyshev polynomials $T_n$. Remember that the polynomial $T_n$ is defined by the property that 
$\Tr\, M^n = T_n (\Tr\,M)$ for every $M\in \SL(\C)$. Applying this to a rotation matrix gives the trigonometric interpretation that $\cos n\theta = \frac 12 T_n (2\cos \theta)$.

\begin{lem}
\label{lem:ChebElementaryProp}
$ $
\begin{enumerate}
\item If $x=a+a^{-1}$, then $T_n(x) = a^n + a^{-n}$;
\item If $y=b+b^{-1}$, the set of solutions to the equation $T_n(x)=y$ consists of the numbers $x=a+a^{-1}$ as $a$ ranges over all $n$--roots of $b$. 
\end{enumerate}
\end{lem}
\begin{proof}
For a matrix $M\in \SL(\C)$, the data of its trace $x$ is equivalent to the data of its spectrum $\{a, a^{-1}\}$. The first property is then a straightfoward consequence of the fact that $\Tr\, M^n = T_n (\Tr\,M)$. The second property immediately follows. 
\end{proof}

We will also need the following quantum trace computation, connecting the skein $[P_k] \in \SSS$ and the central element $H_k\in \ZZ$ that are both associated to the $k$--th puncture of $S$. 

\begin{lem}
\label{lem:QuantumTracePunctureLoop}
For the quantum trace homomorphism $\Tr_\lambda^\omega \colon \SSS \to \ZZ$, 
$$
\Tr_\lambda^\omega \bigl( [P_k] \bigr) = H_k + H_k^{-1}.
$$
\end{lem}

\begin{proof}
Let $e_{i_1}$, $e_{i_2}$, \dots, $e_{i_u}$ be the edges of $\lambda$ that lead to the $k$--th puncture, indexed in counterclockwise order around the puncture; in particular, the $e_{i_j}$ are not necessarily distinct. 

The construction of $\Tr_\lambda^\omega \bigl( [P_k] \bigr)$ in \cite{BonWon1} requires a careful control of elevations (namely $[0,1]$--coordinates) along the knot $P_k \subset S \times [0,1]$. Choose this knot so that it steadily goes up from $e_{i_1}$ to $e_{i_u}$, and then sharply goes down to return to its starting point in $e_{i_1}$. In this setup, the formula of \cite{BonWon1} gives that 
$$
\Tr_\lambda^\omega \bigl( [P_k] \bigr) = 
\omega^{-u+2} Z_{i_1} Z_{i_2} \dots Z_{i_u}
+ \omega^{-u+2} Z_{i_1}^{-1} Z_{i_2}^{-1} \dots Z_{i_u}^{-1}
.
$$
This is relatively straightforward when only one end of the edge $e_{i_1}$ leads to the $k$--th puncture, namely when the projection of $P_k$ to $S$ crosses $e_{i_1}$ only once, but otherwise requires the consideration corrections term in a bigon neighborhood of $e_{i_1}$, of the type given by \cite[Lemma~22]{BonWon1}. Fortunately, these correction terms turn out to be trivial in this case. 

We need to connect this formula to $H_k = [Z_{i_1} Z_{i_2} \dots Z_{i_u}]$. Computing the Weyl quantum ordering is again straightforward when each edge $e_{i_j}$ has only one end leading to the $k$--th puncture. For the general case, we could use a brute force computation as in \cite[Lemma~12]{BonWon4}. We prefer to give here a more indirect argument, based on the invariance of $\Tr_\lambda^\omega \bigl( [P_k] \bigr)$ under isotopy of $P_k$. 

For this, choose the elevation of $P_k$ so that it now goes \emph{down}  from $e_{i_1}$ to $e_{i_u}$, and then  goes up near $e_{i_1}$ to return to its starting point. The formulas of \cite{BonWon1}  give in this setup 
$$
\Tr_\lambda^\omega \bigl( [P_k] \bigr) = 
\omega^{u-2} Z_{i_u} Z_{i_{u-1}} \dots Z_{i_1}
+ \omega^{u-2} Z_{i_u}^{-1} Z_{i_{u-1}}^{-1} \dots Z_{i_1}^{-1}.
$$
Comparing the two expressions for $\Tr_\lambda^\omega \bigl( [P_k] \bigr)\in \ZZ$ shows  in particular that
$$
\omega^{-u+2} Z_{i_1} Z_{i_2} \dots Z_{i_u}
=
\omega^{u-2} Z_{i_u} Z_{i_{u-1}} \dots Z_{i_1}.
$$

By definition of the Weyl quantum ordering, there exists an integer $\alpha \in \Z$ such that 
$$
H_k = \omega^\alpha Z_{i_1} Z_{i_2} \dots Z_{i_u} = \omega^{-\alpha} Z_{i_u} Z_{i_{u-1}} \dots Z_{i_1}.
$$
We can then rephrase the above equality as $\omega^{-\alpha-u+2} H_k = \omega^{\alpha+u-2} H_k$.  Although the current article usually focuses on the case where $\omega$ is a root of unity, these computations are valid for all $\omega$. It follows that $\alpha= -u+2$. This proves that $H_k = \omega^{-u+2} Z_{i_1} Z_{i_2} \dots Z_{i_u}$, and our first computation then shows that $\Tr_\lambda^\omega \bigl( [P_k] \bigr) = H_k + H_k^{-1}$. 
\end{proof}

We are now ready to prove Theorem~\ref{thm:RealizeInvariantsIntro}, which we repeat here for convenience. Recall that a character $r\in \RR$ associates a number $\Tr\,r(P_k)$ to the $k$--th puncture of $S$, where $P_k$ is a small loop going around the puncture. 

\begin{thm}
\label{thm:RealizeInvariants}
Assume that the surface $S$ has at least one puncture, that its Euler characteristic is negative,  that $A$ is a primitive $N$--root of $-1$ with $N$ odd, and that we are given:
\begin{enumerate}
\item a character  $r\in \RR$ realizing some ideal triangulation of $S$;
\item  a number $p_k\in \C$ such that $T_N (p_k) =- \Tr\,r(P_k)$ for each of the punctures of~$S$. 
\end{enumerate} 

Then, there exists an irreducible finite-dimensional representation $\rho \col \SSS \to \End(E)$ whose classical shadow is equal to $r$ and whose puncture invariants are the $p_k$. 
\end{thm}

\begin{proof}
Since $r\in \RR$ realizes the ideal triangulation $\lambda$, Proposition~\ref{prop:RealizeIdealTriang} provides a twisted homomorphism $\zeta\col \mathcal W(\tau_\lambda; \Z) \to \C^*$ and an associated representation $\mu_\zeta \col\ZZZ \to \End(\C)=\C$, such that 
$$
 \mu_\zeta \circ \Tr_\lambda^\iota  \bigl([K]\bigr) = - \Tr\, r(K)
$$
for every framed knot $K \subset S\times[0,1]$

By Lemma~\ref{lem:QuantumTracePunctureLoop}, the image of $[P_k]\in \SSS$ under the quantum trace homomorphism $\Tr^\omega_\lambda \col \SSS \to \ZZ$ is equal to $\Tr_\lambda^\omega\bigl([P_k]\bigr)=H_k + H_k^{-1}$ in $\ZZ$. Similarly, $\Tr_\lambda^\iota\bigl([P_k\bigr])=H_k + H_k^{-1}$ in $\ZZZ$. (Beware that we are using the same symbols to denote the skeins $[P_k] \in \SSS$ and $[P_k] \in \SSSS$, and the central elements $H_k \in \ZZ$ and $H_k \in \ZZZ$.)
Then, for $[P_k] \in \mathcal S^{-1}(S)$, 
$$
\Tr\, r(P_k) =- \mu_\zeta \circ \Tr^\iota_\lambda\bigl([P_k] \bigr) = -\mu_\zeta(H_k + H_k^{-1}) = -g_k -g_k^{-1}
$$
if  we set $g_k =  \mu_\zeta(H_k) \in \End(\C)=\C$. 

For each $k$, we are given a number $p_k\in \C$ such that  $T_N (p_k) =- \Tr\,r(P_k)=g_k +g_k^{-1}$. Lemma~\ref{lem:ChebElementaryProp} then provides an $N$--root $h_k= \sqrt[N]{g_k}$ of such that $p_k = h_k + h_k^{-1}$.

Proposition~\ref{prop:RepsSquareRootCheFock}  associates to the homomorphism $\zeta\col \mathcal W(\tau_\lambda; \Z) \to \C^*$ and to the $N$--roots $h_k= \mu_\zeta(H_k)^{\frac1N}$ an irreducible representation $\mu\col \ZZ \to \End(E)$ such that
\begin{enumerate}
\item $\mu(Z_\alpha^N) = \zeta(\alpha) \,\Id_E$ for every $\alpha \in \mathcal W(\tau_\lambda; \Z)$;
\item $\mu(H_k) = h_k\,\Id_E$ for every $k=1$, \dots, $s$.
\end{enumerate}
Composing with the quantum trace map $\Tr^\omega_\lambda \col \SSS \to \ZZ$, we now have a representation
$$
\rho = \mu \circ \Tr^\omega_\lambda \col \SSS \to \End(E).
$$

Let $K$ be a framed knot whose projection to $S$ has no crossing and whose framing is vertical. Then, for the associated skein $[K] \in \SSS$, 
$$
T_N \bigl( \rho([K])\bigr) = \rho \bigl( T_N([K])\bigr) =  \mu \circ \Tr^\omega_\lambda \bigl( T_N([K])\bigr) =  \mu \circ \Tr^\omega_\lambda \circ  \mathbf T^A \bigl([K]\bigr)
=\mu \circ \mathbf F^\omega \circ \Tr^\iota_\lambda \bigl([K]\bigr)
$$
by using the fact that $\rho$ is an algebra homomorphism for the first equality, by definition of the Chebyshev homomorphism $\mathbf T^A \col \SSSS \to \SSS$  in \S \ref{sect:ChebFrob} for the third equality, and by the miraculous cancellations of  Theorem~\ref{thm:ChebyQTracesFrob} for the last relation. In terms of the Frobenius homomorphism $\mathbf F^\omega \col \TTT   \to \TT$ introduced in \S \ref{sect:ChebFrob} and of the representation $\mu_\zeta \col\ZZZ \to \End(\C)=\C$, the property that $\mu(Z_\alpha^N) = \zeta(\alpha) \,\Id_E =  \mu_\zeta(Z_\alpha)$ for every $\alpha \in \mathcal W(\tau_\lambda; \Z)$ can be rephrased  as $\mu \circ \mathbf F^\omega = \mu_\zeta$.
Therefore, 
$$
T_N \bigl( \rho([K])\bigr)  =  \mu \circ \mathbf F^\omega \circ \Tr^\iota_\lambda \bigl([K]\bigr)
= \mu_\zeta \circ \Tr^\iota_\lambda  \bigl([K]\bigr)\, \Id_E 
= - \Tr\, r(K) \, \Id_E.
$$

Also, for the $k$--th puncture of $S$,
$$
\rho \bigl([P_k] \bigr) = \mu \circ \Tr^\omega_\lambda \bigl( [P_k] \bigr) = \mu(H_k + H_k^{-1}) = (h_k + h_k^{-1})\, \Id_E = p_k\, \Id_E. 
$$

If we knew that  $\rho$ was irreducible, we would be done with the proof of Theorem~\ref{thm:RealizeInvariants}. At this point, there is no reason for this property to hold. However, if $\rho$ is not irreducible, it suffices to consider an irreducible component $\rho' \col \SSS \to \End(F)$ with $F\subset E$. Restricting the above computations to $F$ shows that the classical shadow of the representation $\rho'$ is equal to the character $r\in \RR$,  and that its puncture invariants are equal to the numbers $p_k$. 
 \end{proof}
 
 \begin{rem}
We conjecture that, when $r$ is sufficiently generic in $\RR$, the representation $\rho = \mu \circ \Tr^\omega_\lambda $ used in the proof of Theorem~\ref{thm:RealizeInvariants}  is already irreducible, and that there is no need to restrict to an irreducible factor. We also conjecture that, for generic $r\in \RR$, there is a unique representation $\rho$ satisfying the conclusions of Theorem~\ref{thm:RealizeInvariants}, up to isomorphism. This second conjecture was recently proved by Nurdin Takenov \cite{Tak} for the one-puncture torus and the four-puncture sphere (building on earlier work of Bullock-Przytycki \cite{BullPrz} and Havl\'\i\v cek-Po\v sta \cite{HavPost} for the one-puncture torus). 
\end{rem}

\begin{rem}
\label{rem:ReducibleRep}
In the very non-generic case there $r(P_k)$ is the identity and $p_k = - \omega^4 - \omega^{-4}$ for some punctures, the representation $\rho = \mu \circ \Tr^\omega_\lambda $ is definitely reducible. This is a key ingredient of the ``puncture filling'' process developed in \cite{BonWon4}. 
\end{rem}

\section{A uniqueness property}
\label{sect:UniquenessProp}
We made choices in the proof of Theorem~\ref{thm:RealizeInvariants}. The goal of this section is to show that its output does not depend on these choices, provided we carefully specify our data and our construction. The resulting uniqueness statement will be used in the subsequent article \cite{BonWon4}, where we heavily rely on Theorem~\ref{thm:RealizeInvariants} and apply this statement to suitable punctured surfaces in order to construct representations of the skein algebra of a closed surface.

\subsection{Pleated surfaces and representations of $ \ZZZ$}

The proof of Theorem~\ref{thm:RealizeInvariants} hinges on Proposition~\ref{prop:RealizeIdealTriang} which, given a character $r\in \RR$, provides a twisted homomorphism $\zeta \colon \mathcal W(\tau_\lambda ; \Z) \to \C^*$ and its associated representation $\mu_\zeta \colon \ZZZ \to  \C$ such that 
$$
\mu_\zeta \circ \Tr_\lambda^\iota \bigl( [K] \bigr) = - \Tr\,r(K)
$$
for every framed knot $K \subset S \times [0,1]$. Recall that $\mu_\zeta$ and $\zeta$ are related by the property that $\mu_\zeta (Z_\alpha)=\zeta(\alpha) \in \C^*$ for every basis element $Z_\alpha \in \ZZZ$ associated to an edge weight system $\alpha \in \mathcal W(\tau_\lambda; \Z)$. 

For most characters $r\in \RR$, the homomorphism $\mu_\zeta \colon \ZZZ \to  \C$ is uniquely determined by $r$ and by the pleated surface  $(\widetilde f, \bar r)$. However this uniqueness fails, in a very specific way,  when the character $r\in \RR$ admits a very special type of internal symmetry which we now describe. 

The cohomology group $H^1(S;\Z_2)$ acts on the character variety $\RR$ by the property that, for every homomorphism $r\colon \pi_1(S) \to \SL(\C)$ and cohomology class $\epsilon \in H^1(S;\Z_2)$, the homomorphism $\epsilon r$ is defined by
$$\epsilon r(\gamma) = (-1)^{\epsilon(\gamma)} r(\gamma) \in \SL(\C)$$
for every $\gamma \in \pi_1(S)$. We say that $\epsilon \in H^1(S;\Z_2)$ is a \emph{sign-reversal symmetry} for the character $r\in \RR$ if the action of $\epsilon$ on $\RR$ fixes $r$. This is equivalent to the property that the trace $\Tr\,r(\gamma)$ is equal to $0$ for every $\gamma \in \pi_1(S)$ with $\epsilon(\gamma) \neq 0$. 

The group $H^1(S;\Z_2)$ also acts on the balanced Chekhov-Fock algebra $\ZZ$ by the property that $\epsilon Z_\alpha = (-1)^{\epsilon\left([\alpha] \right)} Z_\alpha$ for every $\epsilon \in H^1(S;\Z_2)$ and every $\alpha \in \mathcal W(\tau_\lambda; \Z)$, where $[\alpha] \in H_1(S;\Z_2)$ is the homology class associated to the edge weight system $\alpha$ as in Lemma~\ref{lem:StructureWeightSystems}.

\begin{prop}
\label{prop:SquareRootsShearBend}
Let the pleated surface $(\widetilde f, \bar r)$ have pleating locus the ideal triangulation $\lambda$, and let $r\in \RR$ be represented by a homomorphism $r\colon \pi_1(S) \to \SL(\C)$ lifting the monodromy $r\colon \pi_1(S) \to \SL(\C)$ of $(\widetilde f, \bar r)$.  Then there exists a homomorphism $\mu_\zeta \colon \ZZZ \to  \C$, associated to a twisted homomorphism $\zeta \colon \mathcal W(\tau_\lambda ; \Z) \to \C^*$, such that 
\begin{enumerate}
\item for each edge $e_i$ of $\lambda$, $\mu_\zeta(Z_i^2)$ is equal to the shear-bend parameter $x_i \in \C^*$ of $e_i$ in the pleated surface $(\widetilde f, \bar r)$;
\item $\mu_\zeta \circ \Tr_\lambda^\iota \bigl( [K] \bigr) = - \Tr\,r(K)$ for every framed knot $K \subset S \times [0,1]$.
\end{enumerate}

In addition, $\mu_\zeta$ is unique up to the action on $\ZZZ$ of an orientation-reversal symmetry $\epsilon \in H^1(S;\Z_2)$ of the character $r\in \RR$.
\end{prop}

We say that a homomorphism $\mu_\zeta \colon \ZZZ \to  \C$ satisfying the above conclusions is \emph{compatible} with the pleated surface $(\widetilde f, \bar r)$  and the character $r\in \RR$.

\begin{proof}[Proof of Proposition~\ref{prop:SquareRootsShearBend}] The existence is provided by Proposition~\ref{prop:RealizeIdealTriang}, or more precisely by its proof to guarantee that $\mu_\zeta(Z_i^2)=x_i$ for every edge $e_i$ of $\lambda$. 

To prove the uniqueness, suppose that we are given another homomorphism $\mu_{\zeta'} \colon \ZZZ \to  \C$ satisfying the same conclusions, and associated to a twisted homomorphism $\zeta' \colon \mathcal W(\tau_\lambda ; \Z) \to \C^*$. From the property that $\mu_{\zeta}(Z_i^2)=\mu_{\zeta'}(Z_i^2)= x_i$, we conclude that $\mu_{\zeta}(Z_\alpha)^2=\mu_{\zeta'}(Z_\alpha)^2$  and therefore $\mu_{\zeta}(Z_\alpha)=\pm \mu_{\zeta'}(Z_\alpha)$ for every $\alpha \in \mathcal W(\tau_\lambda ; \Z)$. Since $\mu_\zeta$ and $\mu_{\zeta'}$ are both algebra homomorphisms, there consequently exists a group homomorphism $\epsilon \colon \mathcal W(\tau_\lambda ; \Z) \to \Z_2$ such that $\mu_{\zeta}(Z_\alpha)=(-1)^{\epsilon(\alpha)} \mu_{\zeta'}(Z_\alpha)$ for every $\alpha \in \mathcal W(\tau_\lambda ; \Z)$. Another application of the property that $\mu_{\zeta}(Z_i^2)=\mu_{\zeta'}(Z_i^2)$ shows that $\epsilon$ is trivial on the subgroup $(2\Z)^n \subset \mathcal W(\tau_\lambda ; \Z)$ of Lemma~\ref{lem:StructureWeightSystems}. This statement then shows that $\epsilon$ comes from a homomorphism $H_1(S; \Z_2) \to \Z_2$, and can therefore be interpreted as a cohomology class $\epsilon \in H^1(S; \Z_2) $. 

In this cohomological interpretation of $\epsilon \in H^1(S; \Z_2) $, we have that $\mu_{\zeta}(Z_\alpha)=(-1)^{\epsilon([\alpha])} \mu_{\zeta'}(Z_\alpha)$ for every $\alpha \in \mathcal W(\tau_\lambda ; \Z)$. Namely, the homomorphisms $\mu_\zeta$, $\mu_{\zeta'} \colon \ZZZ \to  \C$ differ by the action of $\epsilon \in H^1(S;\Z_2)$ on $\ZZZ$.

Given a framed link $K\subset S \times [0,1]$, the construction of the quantum trace $\Tr_\lambda^\iota$  in \cite{BonWon1} shows that  $\Tr_\lambda^\iota \bigl( [K] \bigr)\in \ZZZ$ is a linear combination of monomials $Z_\alpha$ whose associated homology class $[\alpha] \in H_1(S;\Z_2)$, in the sense of Lemma~\ref{lem:StructureWeightSystems}, is the same as the class $[K] \in H_1(S;\Z_2)$ defined by $K$. As a consequence,
\begin{align*}
 \Tr\,r(K) =-\mu_{\zeta'}\circ \Tr_\lambda^\iota  \bigl( [K] \bigr) 
= - (-1)^{\epsilon(K)} \, \mu_\zeta \circ \Tr_\lambda^\iota  \bigl( [K] \bigr)  
= (-1)^{\epsilon(K)}  \Tr\,r(K)
\end{align*}
for every  framed link $K\subset S \times [0,1]$. 
This proves that $\epsilon \in H^1(S;\Z_2)$ is a sign-reversal symmetry for the character $r\in \RR$. 

As a consequence, the homomorphisms $\mu_\zeta$, $\mu_{\zeta'} \colon \ZZZ \to  \C$ differ by the action on $\ZZZ$ of a sign-reversal symmetry $\epsilon \in H^1(S;\Z_2)$ of $r\in \RR$.
\end{proof}

Characters with non-trivial sign-reversal symmetries exist, but are rare. In fact, the characters that have no (non-trivial) sign-reversal symmetries form a Zariski dense closed subset in $\RR$. (Hint: Choose a family of simple closed curves $\gamma_1$, $\gamma_2$, \dots, $\gamma_k$ in $S$ that generate $H_1(S; \Z_2)$, and consider the set of $r\in \RR$ such that $\Tr\, r(\gamma_i) \neq 0$ for some $i$.) This subset of $\RR$ includes all injective homomorphisms $\pi_1(S) \to \SL(\C)$, since their images contain no matrix with trace 0. In particular all ``geometric'' characters, corresponding to fuchsian or quasifuchsian groups, admit no sign-reversal symmetries.. 

More precisely, a simple algebraic manipulation shows that every character $r\in \RR$ with a non-trivial sign-reversal symmetry $\epsilon \in H^1(S; \Z_2)$ is represented by a homomorphism $r\colon \pi_1(S) \to \SL(\C)$ of  the following type: Considering $\epsilon$ as a group homomorphism $\epsilon \colon \pi_1(S) \to \Z_2$ and for an arbitrary $\gamma_0 \in \pi_1(S)$ with $\epsilon(\gamma_0) \neq 0$, there exists a group homomorphism $\theta \colon \ker \epsilon \to \C/2\pi\mathrm i \Z$ such that
$$
r(\gamma_0) = \begin{pmatrix}
\mathrm i & 0 \\ 0 & -\mathrm i
\end{pmatrix}
\text{ and }
r(\gamma) = \begin{pmatrix}
 \cosh \theta(\gamma) & \sinh \theta(\gamma) \\  \sinh \theta(\gamma) & \cosh \theta(\gamma)
\end{pmatrix}
$$
for every $\gamma \in \ker \epsilon$. In particular, noting the constraints that  $\theta(\gamma_0^2)=\pi \mathrm i$ and $\theta(\gamma_0 \gamma \gamma_0^{-1}) =-\theta(\gamma)$ for every $\gamma \in \ker \epsilon$, the space of such characters has complex dimension $2g+s-2$ in the $(6g+3s-6)$--dimensional character variety $\RR$ (where $g$ is the genus of the surface $S$ and $s$ is its number of punctures). 

\subsection{A strengthening of Theorem~\ref{thm:RealizeInvariants}}
Recall that, if the $k$--th puncture of $S$ is adjacent to the edges $e_{i_1}$, $_{i_2}$, \dots, $e_{i_u}$ of the ideal triangulation $\lambda$, it determines an element $H_k= [Z_{i_1}Z_{i_2} \dots Z_{i_u}] \in \ZZZ$. 
\begin{prop}
\label{prop:RealizeInvariantsSpecificForm}

Assume that the surface $S$ has at least one puncture, that its Euler characteristic is negative,  that $A$ is a primitive $N$--root of $-1$ with $N$ odd, and that we are given:
\begin{enumerate}
\item[(i)] a pleated surface $(\widetilde f, \bar r)$ with pleating locus $\lambda$, a character $r\in \RR$ lifting $\bar r\in \mathcal R_{\PSL(\C)}(S)$, and a homomorphism $\mu_\zeta \colon \ZZZ \to \C$ compatible with $(\widetilde f, \bar r)$ and $r$ as in Proposition~{\upshape \ref{prop:SquareRootsShearBend}};
\item[(ii)]  for each puncture of $S$ an $N$--root $h_k$ of $\mu_\zeta(H_k) \in \Z^*$. 
\end{enumerate}
 Then, up to isomorphism, there exists a  unique representation $\mu \colon \ZZ\to \End(E)$ of the balanced Chekhov-Fock algebra $\ZZ$ with the following properties:
\begin{enumerate}
\item the dimension of the vector space $E$ is equal to $N^{3g+s-3} $, where $g$ is the genus of the surface  $S$ and $s$ its number of punctures;
\item $ \mu (Z_\alpha^{N}) = \mu_\zeta(Z_\alpha) $ for every edge weight system $\alpha \in \mathcal W(\tau_\lambda; \Z)$, where we use the same symbol to represent the associated base elements $Z_\alpha \in \ZZ$ and $Z_\alpha \in \ZZZ$;
\item $ \mu(H_k) = h_k \,\Id_E $ for the central element $H_k \in \ZZ$ associated to the $k$--th puncture of $S$.
\end{enumerate}
In addition, the representation $\mu$ is irreducible and the representation $\rho = \mu \circ \Tr_\lambda^\omega \colon \SSS \to \End(E)$ has classical shadow $r \in \RR$, in the sense that 
$$
T_N \bigl( \rho ([K]) \bigr) = - \Tr\, r(K)\, \Id_E
$$
for every knot $K \subset S\times [0,1]$ whose projection to $S$ has no crossing and whose framing is vertical (where $T_N(x)$ is the $N$--th Chebyshev polynomial of the first type). 
\end{prop}

\begin{proof}
The existence and uniqueness part is essentially a restatement of the classification of irreducible representations of $\ZZ$ in Proposition~\ref{prop:RepsSquareRootCheFock}. The fact that $\rho$ has classical shadow $r$ follows from the proof of Theorem~\ref{thm:RealizeInvariants}. 
\end{proof}

Although the representation $\mu \colon \ZZ\to \End(E)$ of Proposition~\ref{prop:RealizeInvariantsSpecificForm} is irreducible, the representation  $\rho = \mu \circ \Tr_\lambda^\omega \colon \SSS \to \End(E)$ is not necessarily irreducible; see Remark~\ref{rem:ReducibleRep}.

\section{The case where $A^N=+1$}
\label{sect:ANplus1}

The case $A^N=+1$ can be deduced from the case $A^N=-1$ by the \emph{Barrett isomorphism} $B_\sigma \colon \SSS \to \mathcal S^{-A}(S)$ associated to a spin structure $\sigma$ on the surface $S$. This isomorphism is defined by the property that, for every framed link $K \subset S \times [0,1]$ with $k$ components, 
$$
B_\sigma \bigl( [K] \bigr) = (-1)^{k+\sigma(K)} [K] \in \mathcal S^{-A}
$$
where $\sigma(K)\in \Z_2$ is the monodromy of the framing of $K$ with respect to $\sigma$. See \cite{Barr} and \cite[\S2]{PrzS} for a proof that $B_\sigma \colon \SSS \to \mathcal S^{-A}(S)$ is an algebra isomorphism.

If $A^N=+1$, an irreducible finite-dimensional representation $\rho \colon \SSS \to \End(E)$ defines an irreducible representation $\rho' = \rho \circ B_\sigma \colon \mathcal S^{-A} (S) \to \End(E)$, to which we can apply Theorems~\ref{thm:InvariantsExist} and \ref{thm:RealizeInvariantsIntro}  since $(-A)^N =-1$ as $N$ is assumed to be odd. 
This process depends on the choice of a spin structure $\sigma$, but we can make if more canonical by the following construction.

 Let  $\Spin$ denote the set of isotopy classes of spin structures on $S$. Any two spin structures differ by an obstruction in $H^1(S;\Z_2)$, which defines an action of $H^1(S;\Z_2)$ on $\Spin$.  The cohomology group $H^1(S;\Z_2)$ also acts on the character variety $\RR$ by the property that, if $ r\in \RR$ and $\alpha \in H^1(S;\Z_2)$, then $\alpha  r\in \RR$ is defined by
$$
\alpha  r(\gamma) = (-1)^{\alpha(\gamma)}  r(\gamma) \in \SL(\C)
$$
for every $\gamma \in \pi_1(S)$. 

The \emph{twisted character variety} $\RS$ is then defined as the quotient
$$
\RS =\left(  \RR \times \Spin \right) /  H^1(S;\Z_2) .
$$

If the twisted character $\widehat r\in \RS$  is represented by $( r, \sigma)\in \RR\times \Spin$ and if $K$ is a framed knot in $S\times [0,1]$, the definition is designed so that the following \emph{trace}
$$
\Tr \, \widehat r(K) = -(-1)^{\sigma(K)} \Tr\, r(K).
$$
depends only on the twisted character $\widehat r\in \RS$, and not on its representative $( r, \sigma)\in \RR\times \Spin$. 

The correspondence $\rho \leftrightarrow \rho\circ  B_\sigma$ is used in \cite{BonWon3} to establish the following result. 

\begin{thm}
[\cite{BonWon3}]
\label{thm:InvariantsExistA^N=+1}
Suppose that $A$ is a primitive $N$--root of $+1$ with $N$ odd, and let $\rho \col \SSS \to \End(E)$ be an irreducible finite-dimensional representation of the Kauffman bracket skein algebra. Let $T_N(x)$ be the $N$-th normalized Chebyshev polynomial  of the first kind. 
\begin{enumerate}
\item  There exists a unique twisted character $\widehat r_\rho \in \RS$ such that
$$
T_N \bigl( \rho ([K]) \bigr) =- \bigl( \Tr\, \widehat r_\rho(K) \bigr) \Id_E
$$
for every framed knot $K\subset S \times [0,1]$ whose projection to $S$ has no crossing and whose framing is vertical. 

\item Let $P_k$ be a small simple loop going around the $k$--th puncture of $S$, and consider it as a knot in $S\times[0,1]$ with vertical framing. Then there exists  a number $p_k\in \C$ such that
$
\rho\bigl([P_k]\bigr) = p_k \Id_E
$.

\item The number $p_k$ of {\upshape (2)} is related to the twisted  character  $\widehat r_\rho \in \RS$ of 
 {\upshape (1)} by the property that  $
T_N (p_k) = -\Tr\, \widehat r_\rho(P_k)
$. \qed
\end{enumerate}
\end{thm}

The same correspondence $\rho \leftrightarrow \rho\circ  B_\sigma$ can be used to prove the following analogue of Theorem~\ref{thm:RealizeInvariantsIntro}.  We say that the twisted character $\widehat r\in \RS$ \emph{realizes} the ideal triangulation $\lambda$ of $S$ if the image $\bar r \in \mathcal R_{\PSL(\C)}(S)$ of $\widehat r$ under the natural projection $\RS \to \mathcal R_{\PSL(\C)}(S)$ realizes $\lambda$ in the sense of \S \ref{sect:ShearCoord}.

\begin{thm}
\label{thm:RealizeInvariantsA^N=+1}
Assume that the surface $S$ has at least one puncture, that its Euler characteristic is negative,   that $A$ is a primitive $N$--root of $+1$ with $N$ odd, and that we are given:
\begin{enumerate}
\item a twisted character  $\widehat r\in \RS$ which realizes some ideal triangulation $\lambda$ of $S$;
\item  a number $p_k\in \C$ such that $T_N (p_k) = - \Tr\, \widehat r(P_k)$ for each of the punctures of~$S$. 
\end{enumerate} 

Then, there exists an irreducible finite-dimensional representation $\rho \col \SSS \to \End(E)$ whose classical shadow is equal to $\widehat r$ and whose puncture invariants are the $p_k$. 
\end{thm}

\begin{proof}
Represent $\widehat r\in \RS$ by a pair $( r, \sigma)\in \RR\times \Spin$. Theorem~\ref{thm:RealizeInvariantsIntro} then provides an irreducible representation $\rho' \colon \mathcal S^{-A}(S) \to \End(E)$ with classical shadow $r\in \RR$ and puncture invariants equal to the $p_k$. Then, $\rho = \rho' \circ B_\sigma \colon \SSS \to \End(E)$ satisfies the required properties. 
\end{proof}

\section*{Appendix. The Thurston intersection form of a train track}

Let $\tau$ be a train track in the surface $S$, and let $\mathcal W(\tau; \Z)$ be the space of integer valued edge weights for $\tau$. Namely, an element $\alpha \in \mathcal W(\tau; \Z)$ assigns a weight $\alpha(e) \in \Z$ to each edge $e$ of $\tau$ in such a way that, at each switch $s$ of $\tau$, the sum of the weight of the edges of $\tau$ coming in on one side of $s$ is equal to the sum of the weights of the edges going out on the other side. This abelian group comes with an additional structure provided by the Thurston intersection form
$$
\Omega \col \mathcal W(\tau; \Z) \times \mathcal W(\tau; \Z) \to \Z
$$
defined as in \S \ref{subsect:AlgStructCheFock}. Namely,
$$
\Omega(\alpha, \beta) =  {\textstyle\frac12}  \kern -10pt \sum_{e \text{ right of }e'}  \bigl( \alpha(e)\beta(e') - \alpha(e')\beta(e) \bigr)
$$
where the sum is over all pairs $(e,e')$ where $e$ and $e'$ are two ``germs of edges'' emerging on the same side of a switch of $\tau$ with $e$ to the right of $e'$ ($e$ and $e'$ are not necessarily adjacent at that switch). At this point, $\Omega(\alpha , \beta)$ is only a half-integer, but Theorem~\ref{thm:StructureThurston} below will prove that it is indeed an integer.

We want to determine the algebraic structure of $\mathcal W(\tau; \Z)$ endowed with $\Omega$.
This is a classical property in the case of real-valued edge weights (see for instance \cite[\S3.2]{PenH} or \cite[\S3]{Bon97}), but the subtleties of the integer-valued case seem less well known.  The result is of independent interest because, beyond  the scope of this article, integer-valued edge weight  do occur in geometric situations where the Thurston intersection form is also relevant. One such instance arises for general pleated surfaces where the pleating locus is allowed to have uncountably many leaves, as opposed to the simpler pleated surfaces considered in \S \ref{sect:ShearCoord}. The  bending of such a pleated surface is measured by an edge weight system valued in $\R/2\pi \Z$ for a train track carrying the pleating locus, and this edge weight system is related to rotation numbers by the Thurston intersection form  \cite{Bon97}.

The complement $S-\tau$ of the train track $\tau$ admits a certain number of ``spikes'', each locally delimited by two edges of $\tau$ that approach the same side of a switch of $\tau$. Thicken $\tau$ to a subsurface $U\subset S$ that deformation retracts to $\tau$. Each component of $U-\tau$ is then an annulus that contains one component of $\partial U$ and a certain number of spikes of $S-\tau$. We can then consider the genus $h$ of $U$, and the number $n_{\mathrm{even}}$ (resp. $n_{\mathrm{odd}}$) of components of $U-\tau$ that contain an even (resp. odd) number of spikes. 

A component $U_1$ of $U-\tau$ that contains an even number $n_1>0$ of spikes of $S-\tau$ determines, up to sign, an element of $\mathcal W(\tau; \Z)$ as follows. The core of $U_1$ is homotopic to a closed curve $\gamma_1$ in $\tau$ that is made up of arcs $k_1$, $k_2$, \dots, $k_{n_1}$, $k_{n_1+1}=k_1$, in this order, such that each arc $k_1$ is immersed in $\tau$ and such that two consecutive arcs $k_i$ and $k_{i+1}$ locally bound a spike of $U_1$ at their common end point. For each edge $e$ of $\tau$, we can then consider
$$
\alpha(e) = \sum_{i=1}^{n_1} (-1)^i \alpha_i(e) \in \Z
$$
where $\alpha_i(e)\in \{0,1,2\}$ is the number of times the arc $k_i$ passes over the edge $e$. Because the signs $(-1)^i$ alternate at the spikes of $U_1$ (using the fact that $n_1$ is even for $i=n_1$), one easily sees that these edge weights $\alpha(e)$ satisfy the switch conditions, and therefore define an edge weight system $\alpha \in \mathcal W(\tau; \Z)$. 

A component $U_1$ of $U-\tau$ that contains no spike similarly determines an edge weight system $\alpha\in \mathcal W (\tau; \Z)$. The core of $U_1$ is now homotopic to a closed curve $\gamma_1$ immersed in $\tau$, and $\alpha$ associates to each edge $e$ the number $\alpha(e)$ of times $\gamma_1$ passes over $e$. 

\begin{thm}
\label{thm:StructureThurston}
For a connected train track $\tau$ in the surface $S$, let the numbers $h$, $n_{\mathrm{even}}$ and $n_{\mathrm{odd}}$ be defined as above. Then, the lattice 
$\mathcal W(\tau; \Z)$ of integer valued edge weight systems for $\tau$ admits a basis in which the Thurston intersection form $\Omega$ is block diagonal with 
\begin{itemize}
\item
 $h$  blocks 
$\begin{pmatrix}
0&1\\-1&0
\end{pmatrix}$,
$\frac12 n_{\mathrm{odd}} -1 $  blocks 
$\begin{pmatrix}
0&2\\-2&0
\end{pmatrix}$, and
 $n_{\mathrm{even}}$ blocks
$\begin{pmatrix}
0
\end{pmatrix}$ if $n_{\mathrm{odd}} >0$;

\item
 $h$  blocks 
$\begin{pmatrix}
0&1\\-1&0
\end{pmatrix}$
 and
 $n_{\mathrm{even}}$ blocks
$\begin{pmatrix}
0
\end{pmatrix}$ if $n_{\mathrm{odd}} =0$ and $\tau$ is non-orientable;

\item
 $h$  blocks 
$\begin{pmatrix}
0&1\\-1&0
\end{pmatrix}$
 and
 $n_{\mathrm{even}}-1$ blocks
$\begin{pmatrix}
0
\end{pmatrix}$ if $n_{\mathrm{odd}} =0$ and $\tau$ is orientable.

\end{itemize}

In addition, in all cases, we can choose the base elements corresponding to the blocks $\begin{pmatrix}
0
\end{pmatrix}$ to be  edge weight systems associated as above to components of $U-\tau$ that contain an even number of spikes. 

\end{thm}

In particular, $n_{\mathrm{odd}}$ is always even.

\begin{proof}

We will subdivide the proof into several lemmas. The reader may recognize many analogies with the arguments used in the proof of \cite[Prop.~5]{BonLiu}.

We first discuss a classical homological interpretation of the elements of $\mathcal W(\tau; \Z)$ and of the Thurston intersection form $\Omega$. 

Because the edges of $\tau$ are not oriented, an edge weight system does not directly define a homology class in $H_1(\tau; \Z)$. Instead consider the $2$--fold orientation covering $\widehat\tau$ of $\tau$, consisting of all pairs $(x,o)$ where $x\in \tau$ and $o$ is a local orientation of the train track $\tau$ at $x$. Note that $\widehat\tau$ is a canonically oriented train track, and that the covering involution $\sigma\col \widehat\tau \to \widehat\tau$ that exchanges the two sheets of the covering reverses the orientation of $\widehat\tau$. 

An edge weight system $\alpha\in  \mathcal W(\tau; \Z) $ lifts to a weight system $\widehat\alpha\in  \mathcal W(\widehat\tau; \Z) $. Endowing each (oriented) edge of $\widehat \tau$ with the weight assigned by $\widehat \alpha$ defines a chain, which is closed because of  the switch condition and therefore defines a homology class $[\widehat \alpha] \in H_1(\widehat\tau; \Z)$. Note that $\sigma_* \bigl( [\widehat\alpha] \bigr) = - [\widehat\alpha]$ since the covering involution $\sigma$ reverses the canonical orientation of $\widehat\tau$. 

Conversely, each homology class  $[\widehat \alpha] \in H_1(\widehat\tau;\Z)$ is represented by a unique linear combination of the edges of $\widehat\tau$, and therefore determines  an edge weight system $\widehat \alpha \in \mathcal W(\widehat\tau; \Z)$. Assuming in addition that $\sigma_* \bigl( [\widehat\alpha] \bigr) = - [\widehat\alpha]$, this edge weight system is invariant under the action of $\sigma$, and therefore comes from an edge weight system $\alpha \in \mathcal W (\tau; \Z)$. This proves:

\begin{lem}
\label{lem:EdgeWeightsHomology}
The above correspondence  identifies the space $\mathcal W (\tau; \Z)$ of edge weight systems  to the eigenspace
$$
H_1(\widehat\tau;\Z)^- = \bigl \{ [\widehat \alpha] \in H_1(\widehat\tau;\Z); \sigma_* \bigl( [\widehat\alpha] \bigr) = - [\widehat\alpha]\bigr\} \subset H_1(\widehat\tau;\Z)
$$
of the homomorphism $\sigma_*\col H_1(\widehat \tau; \Z)\to H_1(\widehat \tau; \Z)$ induced by $\sigma$. \qed
\end{lem}

To describe the Thurston intersection form in this homological framework, consider the subsurface $U$  deformation retracting to $\tau$. The covering $\widehat \tau \to \tau$ uniquely extends to a 2--fold covering $\widehat U \to U$, whose  covering involution $\sigma\col \widehat U \to \widehat U$ extends our previous involution $\sigma$.

\begin{lem}
\label{lem:ThurstonHalfIntersection}
If $[\widehat \alpha]$, $[\widehat\beta] \in H_1(\widehat\tau)^-$ are associated to the edge weight systems $\alpha$, $\beta\in \mathcal W(\tau_\lambda; \Z) $,
$$
\Omega(\alpha, \beta ) ={\textstyle \frac 12 } \, [\widehat\alpha] \cdot [\widehat\beta]
$$
where $\cdot$ denotes the algebraic intersection number of classes  of $H_1(\widehat U;\Z)\cong H_1(\widehat \tau;\Z)$.  In addition, $[\widehat\alpha] \cdot [\widehat\beta]$ is even, and $\Omega(\alpha, \beta)$ is an integer.
\end{lem}

\begin{proof}
 To prove the first statement push  the oriented train track $\widehat \tau$ to its left  to obtain a train track $\widehat\tau'\subset \widehat U$ that is transverse to $\widehat\tau$,   realize the homology class $[\widehat\alpha]$ by $\widehat \tau$ endowed with the edge multiplicities coming from $\alpha$,  realize $[\widehat\beta]$ by $\widehat \tau'$ endowed with the edge multiplicities coming from $\beta$, and use this setup to compute the algebraic intersection number $  [\widehat\alpha] \cdot [\widehat\beta]$. Evaluating  the contribution to $  [\widehat\alpha] \cdot [\widehat\beta]$ of each point of $\widehat\tau \cap \widehat\tau'$ then shows that this algebraic intersection number is equal to $2\Omega(\alpha, \beta)$. 
 
The second statement is obtained by a similar but different computation of $  [\widehat\alpha] \cdot [\widehat\beta]$. Perturb $\tau$ to a train track $\tau''$ that is transverse to $\tau$, and let $\widehat\tau''$ be the pre-image of $\tau''$ in $\widehat U$. Now compute  $  [\widehat\alpha] \cdot [\widehat\beta]$ by realizing the homology class $[\widehat\beta]$ by $\widehat \tau''$ endowed with the edge multiplicities coming from $\beta$, while still realizing $[\widehat\alpha]$ by $\widehat \tau$ endowed with the edge multiplicities coming from $\alpha$. The intersection $\widehat\tau \cap \widehat\tau''$ splits into pairs of points interchanged by the covering involution $\sigma$, and the two points in each pair have the same contribution to $  [\widehat\alpha] \cdot [\widehat\beta]$. It follows that $  [\widehat\alpha] \cdot [\widehat\beta]$ is even.
\end{proof}

 We now need to better understand the action of $\sigma_*$ on the homology group $H_1(\widehat U; \Z)$. 
 
 It will be convenient to systematically use a notation which already appeared in Lemma~\ref{lem:EdgeWeightsHomology}. If $V$ is a space where some restriction of the covering involution $\sigma$ induces a homomorphism $\sigma_*$, then
 $$
 V^- = \{ \alpha \in V; \sigma_*(\alpha) = -\alpha \}. 
 $$
 For instance, Lemma~\ref{lem:EdgeWeightsHomology} provides a natural isomorphism $ \mathcal W(\tau; \Z) \cong H_1(\widehat U; \Z)^- $.

Let $\partial_{\mathrm{even}}U$ be the union of the $n_{\mathrm{even}}$ components of $\partial U$ that are adjacent to an even number of spikes to $S-\tau$, and set $\partial_{\mathrm{odd}}U=\partial U -  \partial_{\mathrm{even}}U$. 

\begin{lem}
\label{lem:EvenSpikes}
Let $\gamma_1$ be a component of $\partial_{\mathrm{even}} U$, and let $\widehat\gamma_1$ be its pre-image in $\widehat U$. Then $H_1(\widehat \gamma_1);\Z)^- \cong \Z$, and the image  in $H_1(\widehat U; \Z)^- \cong \mathcal W(\tau; \Z)$ of one of its generators coincides up to sign with the edge weight system that we associated right before Theorem~{\upshape \ref{thm:StructureThurston}} to the component $U_1$ of $U-\tau$ that contains $\gamma_1$. 
\end{lem}
\begin{proof}
As right above Theorem~\ref{thm:StructureThurston}, the curve $\gamma_1$ is homotopic to a closed curve $\gamma_1'$ in $\tau$ that is made up of $n_1$ arcs $k_1$, $k_2$, \dots, $k_{n_1}$, $k_{n_1+1}=k_1$, in this order, such that each arc $k_1$ is immersed in $\tau$ and such that two consecutive arcs $k_i$ and $k_{i+1}$ locally bound a spike of $U_1$ at their common end point. Because $n_1$ is even, there are two possible ways to orient these arcs in such a way that consecutive arcs have opposite orientations. This shows that $\gamma_1'$ has two distinct lifts to $\widehat\tau$, and therefore that the pre-image $\widehat\gamma_1$ of $\gamma_1$ in $\widehat U$ consists of two components of $\partial \widehat U$ that are exchanged by the covering involution. This provides an isomorphism $H_1(\widehat \gamma_1; \Z) \cong \Z \oplus \Z$ where $\sigma_*$ exchanges the two factors. It immediately follows that $H_1(\widehat \gamma_1; \Z)^- \cong \Z$. 

If $\widehat \gamma_1' \subset \widehat\tau$ denotes one of the two lifts of $\gamma_1'$ to $\widehat\tau$, the  image of $H_1(\widehat \gamma_1; \Z)^-$ in $H_1(\widehat U; \Z)^- \cong H_1(\widehat \tau; \Z)^- \cong \mathcal W(\tau; \Z)$ is generated by $[\widehat\gamma_1']-\sigma_*\bigl([\widehat\gamma_1'] \bigr)$. The second statement easily follows. 
\end{proof}

To prove Theorem~\ref{thm:StructureThurston}, we will first restrict attention to the case where $n_{\mathrm{odd}}>0$. This is equivalent to the property that $\partial_{\mathrm{odd}}U$ is non-empty. 

We just saw that the restriction of the covering $\widehat U \to U$ above $\partial_{\mathrm{even}}  U$ is trivial; similarly, its restriction above each component of $\partial_{\mathrm{odd}}  U$ is non-trivial. 
Therefore, the covering $\widehat U \to U$ is classified by a cohomology class in $H^1(U; \Z_2)$ which evaluates to 0 on the elements of $\partial_{\mathrm{even}}U$ and to 1 on the components of $\partial_{\mathrm{odd}}U$. 

Since the subset  $\partial_{\mathrm{odd}}U$ is non-empty, and  can therefore realize the  cohomology class   classifying the covering $\widehat U \to U$ as the Poincar\'e dual of  a family $K\subset U$ of disjoint arcs whose boundary $\partial K = K\cap \partial U$ consists of one point in each component of  $\partial_{\mathrm{odd}}U$. 

Split $U$ along a separating simple closed curve $\gamma$ to isolate $K$ inside of a planar surface $U_1\subset S$ with boundary $\partial U_1 = \gamma \cup \partial_{\mathrm{odd}}U$, while the closure $U_2$ of $U-U_1$ has genus $h$ and boundary  $\partial U_2 = \gamma \cup \partial_{\mathrm{even}}U$. Let $\widehat U_1$ and $\widehat U_2$ be the respective pre-images of $U_1$ and $U_2$ in $\widehat U$. 

Since $K$ is disjoint from $U_2$, the covering $\widehat U_2 \to U_2$ is trivial, and $\widehat U_2$ consists of two disjoint copies of the surface $U_2$ which are exchanged by $\sigma$.

The covering $\widehat U_1 \to U_1$ is non-trivial above each component of $\partial_{\mathrm{odd}} U$ and trivial above $\gamma$. Since the surface $U_1$ is planar, an Euler characteristic computation shows  that $\widehat U_1$ has genus $\frac12 n_{\mathrm{odd}} -1 $ and has $n_{\mathrm{odd}} +2$ boundary components. 

Consider the Mayer-Vietoris exact sequence
$$
0 \to H_1(\widehat \gamma; \Z) 
\to H_1(\widehat U_1; \Z) \oplus H_1(\widehat U_2; \Z)
\to H_1(\widehat U; \Z) 
\to 0
$$
where $\widehat\gamma$ denotes the pre-image of $\gamma$ in $\widehat U$. (To explain the $0$ on the right, note that the map $H_0(\widehat\gamma;\Z) \to H_0(\widehat U_2; \Z)$ is injective.)

\begin{lem}
\label{lem:ExactSeq}
Remembering that  $V^-$ denotes the $(-1)$--eigenspace of the action of $\sigma_*$ over a space $V$, the above exact sequence induces another exact sequence 
$$
0 \to H_1(\widehat \gamma; \Z) ^-
\to H_1(\widehat U_1; \Z) ^-\oplus H_1(\widehat U_2; \Z)^-
\to  H_1(\widehat U; \Z) ^-
\to 0.
$$
\end{lem}

\begin{proof}
The only point that requires some thought is the fact that the third homomorphism is surjective.

Given $u\in H_1(\widehat U;\Z)^-$, 
the first exact sequence provides $u_1\in H_1(\widehat U_1;\Z)$ and $u_2\in H_1(\widehat U_2;\Z)$ such that $ u =  u_1 +  u_2$ in $H_1(\widehat U;\Z)$. Since $\sigma_*(u)=-u$, we conclude that there exists $v\in H_1(\widehat \gamma; \Z)$ such that $ \sigma_*(u_1)= -u_1 +v$ in $H_1(\widehat U_1;\Z)$ and  $ \sigma_*(u_2)=-u_2 -v$ in $H_1(\widehat U_2;\Z)$. Note that $v\in H_1(\widehat \gamma; \Z)$ is invariant under $\sigma_*$. Therefore, for the isomorphism $H_1(\widehat\gamma; \Z) \cong H_1(\gamma; \Z) \oplus H_1(\gamma; \Z)$ coming from the fact that each of the two components of $\widehat\gamma$ is naturally identified to $\gamma$, $v=(w, w)$ for some $w\in H_1(\gamma; \Z)$. If we replace $u_1$ by $u_1' = u_1 -(w,0)$ and $u_2$ by $u_2'=u_2+(w,0)$, we now have that $u=u_1' + u_2'$ with $ \sigma_*(u_1')= -u_1'$ and $ \sigma_*(u_2')= -u_2'$, as requested. 
\end{proof}

We now analyze the terms of the exact sequence of Lemma~\ref{lem:ExactSeq}.

The space $H_1(\widehat U_2; \Z)^-$ is easy to understand, because $\widehat U_2$ is made up of two disjoint copies of $U_2$, which are exchanged by the covering involution $\sigma$. Therefore,  $H_1(\widehat U_2; \Z)\cong H_1( U_2; \Z) \oplus H_1( U_2; \Z)$ and, for this isomorphism, $H_1(\widehat U_2; \Z)^-$ corresponds to $\{ (\alpha,- \alpha) ; \alpha \in H_1( U_2; \Z)\} $. This defines an isomorphism $H_1(\widehat U_2; \Z)^- \cong H_1( U_2; \Z)$, for which the intersection form of $H_1(\widehat U_2; \Z)^- $ corresponds to twice the intersection form of $ H_1( U_2; \Z)$.

\begin{lem}
\label{lem:ComputeU2}

 There  exists a  basis for $H_1(\widehat U_2; \Z)^- $  in which the intersection form  is block diagonal with $h$ blocks
$
\begin{pmatrix}
0 & -2\\ 2 &0
\end{pmatrix}
$
and $n_{\mathrm{even}}$ blocks
$
\begin{pmatrix}
0
\end{pmatrix}
$. 

In addition, we can arrange that the basis elements corresponding to the blocks
 $
\begin{pmatrix}
0
\end{pmatrix} $ are the images of generators of $H(\widehat\alpha; \Z)^- \cong \Z$ as $\widehat\alpha \subset \widehat U_2$ ranges over all preimages of components $\alpha$ of $\partial_{\mathrm{even}}U$, and that a generator of  $H_1(\widehat \gamma; \Z)^- \cong \Z$ is sent to the sum of these elements.
\end{lem}

\begin{proof} The surface $U_2$ has genus $h$ and has $n_{\mathrm{even}}+1$ boundary components, and $\gamma$ is one of these boundary components. We can therefore find a basis for $H_1(U_2; \Z)$ in which the intersection form  is block diagonal with $h$ blocks
$
\begin{pmatrix}
0 & -1\\ 1 &0
\end{pmatrix}
$
and $n_{\mathrm{even}}$ blocks
$
\begin{pmatrix}
0
\end{pmatrix}
$. In addition, since $\partial U_2 = \gamma \cup \partial_{\mathrm{even}}U$, we can arrange that  the basis elements  corresponding to the blocks 
$
\begin{pmatrix}
0
\end{pmatrix}
$
are the images of generators of $H_1(\alpha; \Z)$ as $\alpha$ ranges over all components of $ \partial_{\mathrm{even}}U$, while the image of a generator of $H_1(\gamma;\Z)$ is sent to the sum of these elements. 

The result then follows by considering the  isomorphism $H_1(\widehat U_2; \Z)^- \cong H_1( U_2; \Z)$ mentioned above.
\end{proof}

We now consider  $H_1(\widehat U_1; \Z)^-$.

\begin{lem}
\label{lem:ComputeU1}
There exists a basis for $H_1(\widehat U_1; \Z)^-$   in which the intersection form is block diagonal with $\frac 12 n_{\mathrm{odd}}-1$ blocks
$
\begin{pmatrix}
0&4\\-4&0
\end{pmatrix}
$
and with one block 
$
\begin{pmatrix}
0
\end{pmatrix}
$. In addition, the block
$
\begin{pmatrix}
0
\end{pmatrix}
$
corresponds to the image of the homomorphism $H_1(\widehat\gamma;\Z)^- \to H_1(\widehat U_1; \Z)^-$ induced by the inclusion map. 
\end{lem}

\begin{proof}
We will use an explicit description of the covering $\widehat U_1 \to U_1$, with a specific basis for $H_1(\widehat U_1; \Z)$. 

Recall that this covering is classified by a cohomology class in $H^1(U_1; \Z_2)$ that is dual to a family  $K\subset U_1$ of  $\frac12 n_{\mathrm{odd}}$ disjoint arcs, with  one boundary point in each component of $\partial_{\mathrm{odd}} U$. Index the components of $\partial_{\mathrm{odd}} U$ as $\alpha_1$, $\alpha_2$, \dots, $\alpha_{n_{\mathrm{odd}}}$ and the components of $K$ as $k_1$, $k_3$, $k_5$, \dots, $k_{n_{\mathrm{odd}}-1}$ in such a way that   $k_{2i-1}$ joins $\alpha_{2i-1}$ to $\alpha_{2i}$. Add to $K$ a family of disjoint arcs $k_2$, $k_4$, \dots, $k_{n_{\mathrm{odd}}-2}$, disjoint from the $k_{2i-1}$, such that each $k_{2i}$ joins $\alpha_{2i}$ to $\alpha_{2i+1}$. See Figure~\ref{fig:U1}.

\begin{figure}[htbp]

\SetLabels
(.5 * .96) $\widehat \gamma$\\
(.5 * .23) $ \gamma$\\
(.095 * .69) $\widehat\alpha_1$\\
(.23 * .69) $\widehat\alpha_2$\\
(.29 * .69) $\widehat\alpha_3$\\
(.425 * .69) $\widehat\alpha_4$\\
( .94* .69) $\widehat\alpha_{n_{\mathrm{odd}}}$\\
(.11 * .62) $\widehat\beta_1$\\
(.23 * .45) $\widehat\beta_2$\\
(.31 * .59) $\widehat\beta_3$\\
(.44 * .45) $\widehat\beta_4$\\
( .8* .45) $\widehat\beta_{n_{\mathrm{odd}}-2}$\\
( .935* .62) $\widehat\beta_{n_{\mathrm{odd}}-1}$\\
(.108 * .128) \tiny $\alpha_1 $ \\
( .203*.128 )\tiny $\alpha_2 $ \\
(.298 *.128 ) \tiny$ \alpha_3$ \\
(.399 *.128 ) \tiny$ \alpha_4$ \\
(.17 * .05) $\beta_1$\\
(.26 * .05) $\beta_2$\\
(.35 * .05) $\beta_3$\\
(.44 * .05) $\beta_4$\\
( .5*.3 )\Large $\downarrow $ \\
(.5 * .69) $\dots $ \\
\endSetLabels
\centerline{\AffixLabels{ \includegraphics{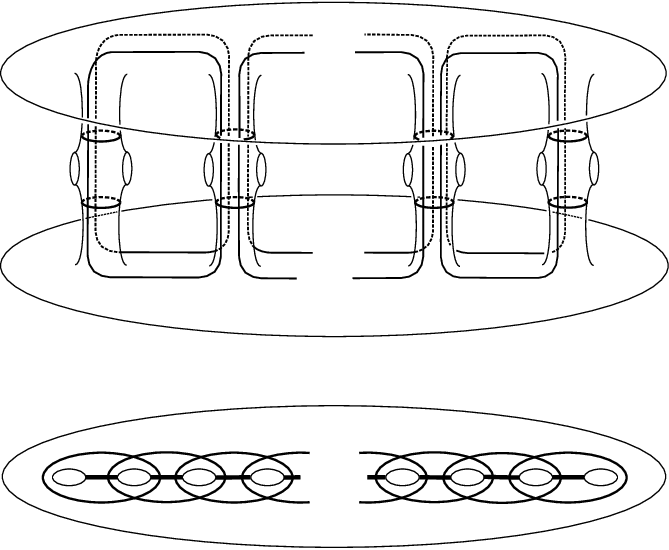}}}

\caption{}
\label{fig:U1}
\end{figure}

For $i=1$, 2, \dots, $n_{\mathrm{odd}}-1$, consider a small regular neighborhood of $k_i \cup \alpha_i \cup \alpha_{i+1}$ in $U_1$ and let $\beta_i$ be the boundary component of this neighborhood which is neither $\alpha_i$ nor $\alpha_{i+1}$; endow $\beta_i$ by the corresponding boundary orientation. Orient each curve $\alpha_i$ by the boundary orientation of $\partial_{\mathrm{odd}} U$. 

The pre-image of each curve $\alpha_i$ is a single curve $\widehat \alpha_i$, which we orient by the orientation of $\alpha_i$. 
The pre-image of $\beta_j$  in $\widehat U_1$ consists of two disjoint curves. Arbitrarily choose one of these curves $\widehat \beta_j$ and orient it by the orientation of $\beta_j$.  Then, the $[\widehat \alpha_i]$ and $[\widehat \beta_j]$   form a basis for $H_1(\widehat U_1; \Z)$.  See Figure~\ref{fig:U1}. 

Consider an element $u\in H_1(\widehat U_1;\Z)$, uniquely expressed in this basis as 
$$
u = \sum_{i=1}^{n_{\mathrm{odd}}} a_i [\widehat \alpha_i] + 
 \sum_{j=1}^{n_{\mathrm{odd}}-1} b_j [\widehat \beta_j] 
$$
with all $a_i$, $b_j\in \Z$. By construction of the curves $\widehat \alpha_i$ and $\widehat \beta_j$, 
$$
\sigma_* \bigl( [\widehat \alpha_i]\bigr) =  [\widehat \alpha_i] 
\text{ and }
\sigma_* \bigl( [\widehat \beta_j]\bigr) = -[\widehat \beta_j]  - [\widehat \alpha_j] - [ \widehat \alpha_{j+1}].
$$

If $u$ belongs to $ H_1(\widehat U_1;\Z)^-$, namely if $\sigma_*(u)=-u$, it follows from these observations and from the consideration of the coefficients of each $[\widehat \alpha_i]$ that we necessarily have 
\begin{align*}
b_1 &= 2 a_1,\\
b_i +b _{i-1}  &= 2a_i \text { for every }i \text { with } 2\leq i \leq n_{\mathrm{odd}}-1,\\
\text{and }
b_{n_{\mathrm{odd}}-1} &= 2 a_{n_{\mathrm{odd}}}.
\end{align*}
In particular, the coefficients $b_j$ are all even, and
$$
u = \frac{u-\sigma_*(u)}2 =  \sum_{j=1}^{n_{\mathrm{odd}}-1} \frac{b_j}2  \bigl ( [\widehat \beta_j] - \sigma_*( [\widehat \beta_j] ) \bigr). 
$$
Therefore, the elements $ [\widehat \beta_j] - \sigma_* \bigl ( [\widehat \beta_j] \bigr)$ generate $ H_1(\widehat U_1;\Z)^-$. Since these elements $ [\widehat \beta_j] - \sigma_* \bigl( [\widehat \beta_j] \bigr) = 2 [\widehat \beta_j] + [\widehat\alpha_j] + [\widehat\alpha_{j+1}]$ are linearly independent, they form a basis for $ H_1(\widehat U_1;\Z)^-$.

Note that $[\widehat \beta_j] \cdot [\widehat \beta_{j'}] =0$ if $| j-j'|>1$, and $[\widehat \beta_j] \cdot [\widehat \beta_{j+1}] = \epsilon_j = \pm1$, where the sign depend on which lift of $\beta_j$ we chose for $\widehat \beta_j$. Also, 
$$
\sigma_* \bigl( [\widehat \beta_j]\bigr) \cdot [\widehat\beta_{j'}] 
= [\widehat\beta_{j}] \cdot \sigma_* \bigl( [\widehat \beta_{j'}]\bigr) 
= - \sigma_* \bigl( [\widehat \beta_{j}]\bigr)  \cdot  \sigma_* \bigl( [\widehat \beta_{j'}]\bigr) 
=-[\widehat\beta_{j}]  \cdot [\widehat\beta_{j'}] .
$$

It follows that, in the basis of $ H_1(\widehat U_1;\Z)^-$ formed by the $ [\widehat \beta_j] - \sigma_* \bigl( [\widehat \beta_j] \bigr)$, the intersection form has matrix
$$
\begin{pmatrix}
0& 4\epsilon_1 & 0 & 0&\dots &0&0\\
-4\epsilon_1& 0 & 4\epsilon_2 &0& \dots &0&0\\
0& -4\epsilon_2 & 0 & 4\epsilon_3& \dots &0&0\\
0& 0 & -4\epsilon_3& 0 & \dots &0&0\\
\dots &\dots &\dots &\dots &\dots &\dots &\dots \\
0& 0 & 0 & 0& \dots &0&4\epsilon_{n_{\mathrm{odd}}-2}\\
0& 0 & 0& 0 & \dots &-4\epsilon_{n_{\mathrm{odd}}-2} &0
\end{pmatrix}
$$
By block diagonalizing this matrix, a final modification of the basis provides a new basis for $ H_1(\widehat U_1;\Z)^-$ in which the intersection form is block diagonal with $\frac 12 n_{\mathrm{odd}}-1$ blocks
$
\begin{pmatrix}
0&4\\-4&0
\end{pmatrix}
$
and with one block 
$
\begin{pmatrix}
0
\end{pmatrix}
$. 

There remains to show that the block 
$
\begin{pmatrix}
0
\end{pmatrix}
$
corresponds to the image of $H_1(\widehat \gamma; \Z)^-$. This could be seen by explicitly analyzing the block diagonalization process of the above matrix. However, it is easier to note that $H_1(\widehat \gamma; \Z)^- \cong \Z$ is generated by $[\widehat \gamma_1] - [\widehat \gamma_2]$, where $\widehat \gamma_1$ and $\widehat \gamma_2$ are the two components of the pre-image $\widehat\gamma$ of $\gamma$ and are oriented by the boundary orientation of $\partial \widehat U_1$. Then, $[\widehat \gamma_1] - [\widehat \gamma_2]$ is in the kernel of the intersection form of $H_1(\widehat U_1; \Z)^-$, since $\widehat \gamma_1$ and $\widehat \gamma_2$ are in the boundary of $\widehat U_1$, and generate this kernel since it is isomorphic to $\Z$ and since $[\widehat \gamma_1] - [\widehat \gamma_2]$ is indivisible in $H_1(\widehat U_1; \Z)$. 
\end{proof}

We now only need to combine the computations of Lemmas~\ref{lem:ExactSeq}, \ref{lem:ComputeU2} and  \ref{lem:ComputeU1}  to obtain a basis of $ H_1(\widehat U; \Z) ^-$ in which the intersection form is block diagonal with $h$ blocks
$
\begin{pmatrix}
0 & -2\\ 2 &0
\end{pmatrix}
$, $\frac12 n_{\mathrm{odd}}-1$ blocks 
$
\begin{pmatrix}
0 & -4\\ 4 &0
\end{pmatrix}
$,
and $n_{\mathrm{even}}$ blocks
$
\begin{pmatrix}
0
\end{pmatrix}
$. 

Applying Lemmas~\ref{lem:EdgeWeightsHomology} and \ref{lem:ThurstonHalfIntersection} to connect this to the Thurston intersection form on the edge weight space $\mathcal W(\tau;\Z)$, we conclude that $\mathcal W(\tau;\Z)$ admits a basis in which the intersection form is block diagonal with $h$ blocks
$
\begin{pmatrix}
0 & -1\\ 1 &0
\end{pmatrix}
$, $\frac12 n_{\mathrm{odd}}-1$ blocks 
$
\begin{pmatrix}
0 & -2\\ 2 &0
\end{pmatrix}
$,
and $n_{\mathrm{even}}$ blocks
$
\begin{pmatrix}
0
\end{pmatrix}
$. In addition, by the second half of Lemma~\ref{lem:ComputeU2} and using Lemma~\ref{lem:EvenSpikes}, the generators corresponding to the blocks $
\begin{pmatrix}
0
\end{pmatrix}
$ 
can be assumed to correspond to the elements of $\mathcal W(\tau;\Z)$ associated to the components of $\partial_{\mathrm{even}}U$. 

This proves Theorem~\ref{thm:StructureThurston}, under our assumption that $n_{\mathrm{odd}}>0$. 

\medskip 

We now consider the case where $n_{\mathrm{odd}}=0$, namely where $\partial _{\mathrm{odd}} U = \varnothing$, and where the train-track $\tau$ is non-orientable. This second property is equivalent to the property that the covering $\widehat U \to U$ is non-trivial. We can then realize the cohomology class of $H^1(U;\Z_2)$ classifying the covering $\widehat U \to U$ as the Poincar\'e dual of a non-separating simple closed curve $K$. 
Let $U_1 \subset U$ be a surface of genus 1 containing $K$ and bounded by a simple closed curve $\gamma$, and let $U_2$ be the closure of $U-U_1$. As before, let $\widehat U_1$, $\widehat U_2$, $\widehat\gamma$ denote the respective pre-images of $U_1$, $U_2$, $\gamma$ in $\widehat U$. 

The computation of Lemma~\ref{lem:ComputeU2} applies to this case as well, and provides a basis for $H_1(\widehat U_2; \Z)^- $   in which the intersection form  is block diagonal with $h$ blocks
$
\begin{pmatrix}
0 & -2\\ 2 &0
\end{pmatrix}
$
and $n_{\mathrm{even}}$ blocks
$
\begin{pmatrix}
0
\end{pmatrix}
$. 

The surface $\widehat U_1$ is a twice-punctured torus. A simple analysis of the covering $\widehat U_1 \to U_1$ shows that $H_1(\widehat U_1; \Z)^-\cong \Z$ is equal to the image of 
$H_1(\widehat \gamma; \Z)^-$. The intersection form of  $H_1(\widehat U_1; \Z)^-$ is then $0$. 

Again, combining these computations with the exact sequence
$$
0 \to H_1(\widehat \gamma; \Z) ^-
\to H_1(\widehat U_1; \Z) ^-\oplus H_1(\widehat U_2; \Z)^-
\to  H_1(\widehat U; \Z) ^-
\to 0.
$$
provides in this case a basis for $ H_1(\widehat U; \Z) ^-$ in which the intersection form is block diagonal with $h$ blocks
$
\begin{pmatrix}
0 & -2\\ 2 &0
\end{pmatrix}
$
and $n_{\mathrm{even}}$ blocks
$
\begin{pmatrix}
0
\end{pmatrix}
$. 
Using Lemmas~\ref{lem:EdgeWeightsHomology} and \ref{lem:ThurstonHalfIntersection}, this provides the result promised in Theorem~\ref{thm:StructureThurston} in this case as well. The fact that the generators corresponding to the blocks
$
\begin{pmatrix}
0
\end{pmatrix}
$
can be chosen to be the elements associated to the components of $\partial_{\mathrm{even}}U$ is a byproduct of the proof as in the previous case. 

\medskip

Finally, we need to consider the case where  $n_{\mathrm{odd}}=0$ and the train-track $\tau$ is orientable. Then the covering $\widehat U \to U$ is trivial, so that $H_1(\widehat U; \Z)^- \cong H_1(U;\Z)$ in such a way that the intersection form of $H_1(\widehat U; \Z)^-$ corresponds to twice the intersection form of $H_1(U; \Z)$. By Lemma~\ref{lem:ThurstonHalfIntersection}, the last case of Theorem~\ref{thm:StructureThurston} immediately follows. 
\end{proof}


\begin{thebibliography}{XXXX}

\newcommand{\bibsub}[1]{\kern-0pt${}_{#1}$\kern -0.5pt}

\bibitem[AbF\bibsub1]{AbdFroh1} Nel Abdiel, Charles Frohman, \emph{ Frobenius algebras derived from the Kauffman bracket skein algebra}, preprint, 2014, \texttt{arXiv:1412.4144}.

\bibitem[AbF\bibsub2]{AbdFroh2} Nel Abdiel, Charles Frohman, \emph{The localized skein algebra is Frobenius}, preprint, 2015, \texttt{arXiv:1501.02631}.

\bibitem[Ba]{Barr} John W. Barrett, \emph{Skein spaces and spin structures}, Math. Proc. Cambridge Philos. Soc.  126  (1999), 267--275.

\bibitem[BHMV]{BHMV}
Christian Blanchet, Nathan Habegger, Gregor Masbaum, Pierre Vogel, \emph{Topological quantum field theories derived from the {K}auffman
  bracket}, Topology 34 (1995), 883--927.
  
\bibitem[Bo]{Bon97} Francis Bonahon,
\emph{Shearing hyperbolic
surfaces, bending pleated
surfaces and Thurston's
symplectic form},  Ann. Fac. Sci.
Toulouse Math.   5  (1996), 
233--297. 


\bibitem[BoL]{BonLiu} Francis Bonahon, Xiaobo Liu, \emph{Representations
of
the quantum Teichm\"uller
space and invariants of
surface diffeomorphisms}, Geom. Topol. 11 (2007), 889--938. 

\bibitem[BoW\bibsub1]{BonWon1}  Francis Bonahon, Helen Wong, \emph{Quantum traces for representations of surface groups in $\SL(\C)$}, Geometry \& Topology 15 (2011), 1569--1615.

\bibitem[BoW\bibsub2]{BonWon2} Francis Bonahon, Helen Wong, \emph{Kauffman brackets, character varieties and triangulations of surfaces}, in \emph{Topology and Geometry in Dimension Three: Triangulations, Invariants, and Geometric Structures} (W. Li, L. Bartolini,  J. Johnson, F.  Luo,  R. Myers,  J. H. Rubinstein eds.), Contemporary Mathematics 560, American Math.  Society, 2011. 


\bibitem[BoW\bibsub3]{BonWon3} Francis Bonahon, Helen Wong, \emph{Representations of the Kauffman bracket skein algebra I: invariants and miraculous cancellations}, Invent. Math. (2015), published electronically, DOI 10.1007/s00222-015-0611-y. 


\bibitem[BoW\bibsub4]{BonWon4} Francis Bonahon, Helen Wong, \emph{Representations of the Kauffman bracket skein algebra III: closed surfaces and naturality}, submitted for publication, \texttt{arXiv:1505.01522}. 

\bibitem[BoW\bibsub5]{BonWon5} Francis Bonahon, Helen Wong, \emph{The Witten-Reshetikhin-Turaev representation of the Kauffman bracket skein algebra}, Proc. Amer. Math. Soc. (2015), published electronically, DOI 10.1090/proc/12927. 

\bibitem[BoW\bibsub6]{BonWon6} Francis Bonahon, Helen Wong, \emph{Representations of the Kauffman bracket skein algebra IV: naturality for punctured surfaces}, in preparation. 

\bibitem[Bu\bibsub1]{Bull1} Doug Bullock, \emph{Estimating a skein module with $\SL(\C)$ characters}, Proc. Amer. Math. Soc. 125 (1997), 1835--1839.

\bibitem[Bu\bibsub2]{Bull2} Doug Bullock, \emph{Rings of $\SL(\C)$--characters and the Kauffman bracket skein module},  Comment. Math. Helv.  72  (1997),   521--542. 

\bibitem[BFK\bibsub1]{BFK1} Doug Bullock, Charles Frohman, Joanna Kania-Bartoszy\'nska, \emph{Understanding the Kauffman beacket skein module},  J. Knot Theory Ramifications  8  (1999),  265--277. 


\bibitem[BFK\bibsub2]{BFK2} Doug Bullock, Charles Frohman, Joanna Kania-Bartoszy\'nska, \emph{The Kauffman bracket skein as an algebra of observables},   Proc. Amer. Math. Soc.  130  (2002), 2479--2485.

\bibitem[BuP]{BullPrz} Doug Bullock, J\'ozef H. Przytycki, \emph{Multiplicative structure of Kauffman bracket skein module quantizations}, Proc. Amer. Math. Soc.  128  (2000), 923-931. 



\bibitem[ChF\bibsub1]{CheFoc1} Leonid O.
Chekhov, Vladimir V. Fock,
 \emph{Quantum Teichm\"uller spaces},
Theor.Math.Phys. 120 (1999) 1245--1259. 

\bibitem[ChF\bibsub2]{CheFoc2} Leonid O.
Chekhov, Vladimir V. Fock,
 \emph{Observables in 3D gravity
and geodesic algebras}, in:
\emph{Quantum groups and
integrable systems} (Prague,
2000),  Czechoslovak J. Phys.  50 
(2000),   1201--1208.

\bibitem[Fo]{Foc} Vladimir V. Fock, 
 \emph{Dual Teichm\"uller
spaces},
unpublished preprint,
1997, \texttt{arXiv:Math/dg-ga/9702018} . 

\bibitem[HaP]{HavPost} Miloslav Havl\'\i\v cek, Severin Po\v sta, \emph{On the classification of irreducible finite-dimensional representations of $\mathrm U_q'(\mathrm{so}_3)$ algebra}, J. Math. Physics 42 (2001), 472--500. 

\bibitem[Ka]{Kash} Rinat Kashaev, 
\emph{Quantization of
Teichm\"uller spaces and the
quantum dilogarithm},  Lett. Math.
Phys.  43  (1998),  105--115.

\bibitem[L\^e]{Le} Thang T. Q. L\^e, \emph{On Kauffman bracket skein modules at root of unity},  Alg. Geom. Topol. 15 (2015), 1093--1117. 

\bibitem[Liu]{Liu} Xiaobo Liu,
\emph{The quantum Teichm\"uller
space as a non-commutative
algebraic object}, J. Knot Theory Ramifications 18 (2009), 705--726.

\bibitem[MFK]{Mum} David Mumford, John Fogarty, Frances Kirwan, \emph{Geometric invariant theory. Third edition},  Ergebnisse der Mathematik und ihrer Grenzgebiete  34, Springer-Verlag, 1994. 

\bibitem[PeH]{PenH}
Robert C. Penner, John L. Harer,
\emph{Combinatorics of train
tracks}, Annals of Mathematics
Studies vol. \textbf{125},
Princeton University Press,
Princeton, 1992. 

\bibitem[PrS]{PrzS} Jozef H. Przytycki, Adam S. Sikora, \emph{On skein algebras and $\SL(\C)$-character varieties},
Topology 39 (2000), 115--148. 

\bibitem[ReT] {ReshTur} Nicolai Y. Reshetikhin and Vladimir G. Turaev, \emph{Invariants of $3$--manifolds via link polynomials and quantum groups}, Invent. Math. 103 (1991), 547--597.

\bibitem[Ta]{Tak} Nurdin Takenov, \emph{Representations of the Kauffman skein algebra of small surfaces}, preprint, 2015, \texttt{arXiv:1504.04573}. 

\bibitem[Th]{Thu} William P. Thurston, \emph{The geometry and topology of  $3$--manifolds}, Princeton lecture notes, 1978--1981

\bibitem[Tu\bibsub1]{Tur} Vladimir Turaev, \emph{Skein quantization of Poisson algebras of loops on surfaces},  Ann. Sci. \'Ecole Norm. Sup.   24  (1991), 635--704.

\bibitem[Tu\bibsub2]{TuraevBook}
Vladimir~G. Turaev.
 {\em Quantum invariants of knots and 3-manifolds}, 
  de Gruyter Studies in Mathematics Vol. 18, 
 Walter de Gruyter \& Co., Berlin, 1994.

\bibitem[Wit]{Witten} Edward Witten, \emph{Quantum field theory and the Jones polynomial},  Commun. Math. Phys. 121 (1989), 351--399.


\end{thebibliography}
\end{document}